\documentclass{amsart}
\usepackage{color}
\usepackage{amsmath}
\usepackage{amssymb}
\usepackage{amsthm}
\usepackage{amscd}
\usepackage{eucal} % for \mathcal 
\usepackage{enumitem}
\usepackage{tikz}
\usepackage{hyperref}
\usepackage[utf8]{inputenc}
\usepackage[all]{xy}

\usepackage{hyperref}
\hypersetup{
  colorlinks   = true,          %Colors links instead of ugly boxes
  urlcolor     = blue,          %Color fore  external hyperlinks
  linkcolor    = blue,          %Color of internal links
  citecolor   = red             %Color of citations
}

\theoremstyle{plain}

\newtheorem{theorem}{Theorem}[section]
\newtheorem{conjecture}[theorem]{Conjecture}
\newtheorem{lemma}[theorem]{Lemma}
\newtheorem{proposition}[theorem]{Proposition}

\theoremstyle{definition}
\newtheorem{definition}[theorem]{Definition}
\newtheorem{remark}[theorem]{Remark}

\numberwithin{equation}{section}
\newcommand\fantome[1]{}

\def\Fq{\mathbb F_q}

\def\bC{\mathbb C}
\DeclareMathOperator{\sgn}{sgn}

\newcommand{\ZZ}{\mathbb{Z}}

\newcommand{\FF}{\mathbb{F}}
\newcommand{\CC}{\mathbb{C}}

\newcommand{\EE}{\mathbb{E}}

\newcommand{\NN}{\mathbb{N}}
\newcommand{\PP}{\mathbb{P}}
\newcommand{\TT}{\mathbb{T}}
\newcommand\bb[1]{\mathbb{#1}}

\author{Kwun Chung}
\address{
Department of Mathematics\\
University of California, San Diego (UCSD)\\
9500 Gilman Drive 0112\\
La Jolla, CA  92093-0112\\
United States of America (USA)
}
\email{k7chung@ucsd.edu}

\author{Tuan Ngo Dac}
\address{
Univ Lyon, CNRS, Universit\'e Claude Bernard Lyon 1, 
Institut Camille Jordan, UMR 5208,
F-69622 Villeurbanne, France
}
\email{ngodac@math.univ-lyon1.fr}

\author{Federico Pellarin}
\address{
Institut Camille Jordan, UMR 5208,
Site de Saint-Etienne, 23 rue du Dr. P. Michelon,
F-42023 Saint-Etienne, France
}
\email{federico.pellarin@univ-st-etienne.fr}

\title[Universal families of multiple zeta values]{Universal families of Eulerian multiple zeta values in positive characteristic}

\date{\today}

\begin{document}

\begin{abstract}
We study positive characteristic multiple zeta values associated to general curves over $\FF_q$ together with an $\FF_q$-rational point $\infty$ as introduced by Thakur. For the case of the projective line these values were defined as analogues of classical multiple zeta values. In the present paper we first establish a general non-commutative factorization of exponential series associated to certain lattices of rank one. Next we introduce universal families of multiple zeta values of Thakur and show that they are Eulerian in full generality. In particular, we prove a conjecture of Lara Rodr\'iguez and Thakur in \cite{LRT20}.
One of the main ingredients of the proofs is the notion of $L$-series in Tate algebras introduced by the third author \cite{Pel12} in 2012. 
\end{abstract}

\maketitle

%\tableofcontents

\section*{Introduction}

\subsection{Classical multiple zeta values} ${}$\par

Multiple zeta values are positive real numbers that have been involved in certain investigations by Euler in the late eighteenth century. They generalize the values of Riemann's zeta function at integers $\geq 2$ and occur naturally in the computation of their products. Surprisingly, these particular real numbers are ubiquitous in several mathematical and physical theories. For instance, they are connected with periods of mixed Tate motives, and with values of Feynman integrals in perturbative quantum field theory. It is also known that they are related to a certain universal vector bundle with connection on the projective line with the points $0,1,\infty$ removed. These, and other properties, made multiple zeta values the center of intensive studies, especially in the last three decades. We refer the reader to the survey of Zagier \cite{Zag94} and the book of Burgos Gil and Fres\'an \cite{BGF} for a detailed introduction to these topics, as well as for further references. 

Let $\mathbb N$ be the set of positive integers. By definition, the classical multiple zeta values are the values of the convergent series
	\[ \zeta(n_1,\dots,n_r)=\sum_{k_1>\dots>k_r>0} \frac{1}{k_1^{n_1} \dots k_r^{n_r}}, \quad \text{where } n_i\in\NN, \quad n_1 \geq 2. \]
Here $r$ is called the {\em depth} and $w=n_1+\dots+n_r$ the {\em weight} of the presentation $\zeta(n_1,\dots,n_r)$. For notational convenience, we set
$\zeta(\emptyset)=1$. For $r=1$ and $n \geq 2$ we recover the special values $\zeta(n)$ of the Riemann zeta function. 

In 1735, Euler proved that
\begin{equation} \label{Euler}
\frac{\zeta(2n)}{(2\pi)^{2n}}=(-1)^{n+1}\frac12\frac{B_{2n}}{(2n)!} \quad \text{ for all }  n\geq 1, 
\end{equation}

where for $k\geq 0$, $B_k$ denotes the $k$-th Bernoulli number. In view of this, it is natural to classify tuples $(n_1,\dots,n_r)$ with $n_i \geq 1, n_1 \geq 2$ such that the quotient $\zeta(n_1,\dots,n_r)/\pi^{n_1+\dots+n_r}$ is rational. In this case, following 
 \cite{LRT14,CPY19}, we say that the multiple zeta value $\zeta(n_1,\dots,n_r)$ is {\it Eulerian}. 
 
 For example $\zeta(\underbrace{2,\dots,2}_n)$ with $n \geq 1$ is Eulerian. This follows from the identity
\begin{equation} \label{Zagier}
\sum_{n \geq 0} (-1)^n \zeta(\underbrace{2,\dots,2}_n) x^{2n+1}=\frac{\sin(\pi x)}{\pi}=	\sum_{n \geq 0} (-1)^n \frac{\pi^{2n}}{(2n+1)!} x^{2n+1},
\end{equation}
an elegant proof of which is given by Zagier in \cite{Zag12}. Further, one can show that $\zeta(\underbrace{3,1,	\dots,3,1}_n)$ is Eulerian (conjectured in \cite{Zag94} and proved in \cite{BBBL01}).  We refer the reader to \cite[Remark after Conjecture 4.3]{LRT14} and \cite[\S 7.5]{Tha17} for more details and more complete references on Eulerian multiple zeta values. Brown formulated a sufficient (and conjecturally necessary) condition for Eulerian multiple zeta values in terms of motivic multiple zeta (see \cite[Theorem 3.3]{Bro12}). This condition is not completely effective (see \cite[\S 1]{CPY19} for more details).

Finally, we mention that Siegel and Klingen were able to extend Euler's formulas \eqref{Euler} for Dedekind zeta values $\zeta_F(.)$ attached to a number field $F$ and showed that for a totally real number field $F$ and all even $n \geq 2$, $\zeta_F(n)$ is an algebraic multiple of $\pi^{n [F:\mathbb Q]}$ (see \cite{Sie80} for further details).

\subsection{Multiple zeta values in positive characteristic } \label{char p MZV} ${}$\par

There is a deep analogy between the arithmetic of number fields and that of global function fields of positive characteristic. Through this analogy one can track similarities in the comparison of the arithmetic over the ring $\ZZ$ on one side, and the arithmetic over the ring $A$ of rational functions over a given smooth curve $X$ over a finite field, which are regular away from a distinguished point (see below for a rigorous description of the settings). Carlitz, in the years 1930s, inaugurated this study in the case of $X$ the projective line (the genus zero case). As a natural consequence, a theory of positive characteristic multiple zeta values associated to $A$ emerged, pioneered by Thakur. At this level of generality ($X$ a curve over a finite field), the description of the algebraic relations connecting multiple zeta values offers new challenges that do not seem to be immediately visible in the classical setting. It is our aim to shed new light and perspectives on these aspects.

Throughout this paper, let $\Fq$ be the finite field with $q$ elements and characteristic $p>0$. Let $K$ be the function field of a geometrically connected smooth projective curve $X$ over $\Fq$ and let $\infty$ be a place of $K$ of degree $d_\infty=1$.  Let $g$ be the genus of $X$ and let 
\begin{equation}\label{A}
A=H^0(X\setminus\{\infty\},\mathcal{O}_X)
\end{equation} 
be the ring of elements of $K$ which are regular outside $\infty$ ($A$ is called the {\em base ring}).  The $\infty$-adic completion $K_\infty$ of $K$ is equipped with the normalized additive $\infty$-adic valuation $v_\infty:K_\infty \rightarrow \mathbb Z\cup\{+\infty\}$.  Let us set $\deg(\cdot):=-v_\infty(\cdot)$.  The completion $\mathbb C_\infty$ of a fixed algebraic closure $\overline K_\infty$ of $K_\infty$ comes equipped with a unique valuation extending $v_\infty$, which will be again denoted by $v_\infty$.  We fix a uniformizer $\pi$ of $K_\infty$ so that we can identify $K_\infty$ and the field of formal Laurent series 
$\FF_q((\pi))$. This choice also allows to introduce a sign function $\sgn:K_\infty^\times \to \Fq^\times$.
This is the unique group homomorphism defined by $\sgn(\pi)=1$. We define $A_+$ to be the set of monic elements of $A$, i.e., the set of $a \in A$ such that $\sgn(a)=1$. For all $i \in \mathbb N$, we also set $A_{+,i}=\{a \in A_+: \deg(a)=i\}$. 

%With the valuation $v_\infty$, we can define the following norm on $\mathbb C_\infty$: $|x|=q^{-v_\infty(x)}$. We define the Frobenius map $\tau: \mathbb C_\infty \rightarrow \mathbb C_\infty$ as the $\Fq$-algebra homomorphism which sends $x$ to $x^q$.

%\medskip
\subsection*{The genus 0 case}\label{in-genus-0}

We consider the genus $0$ case where $X=\mathbb P^1$ and $\infty$ is an $\FF_q$-rational point on it. Then $A=\Fq[\theta]$  the polynomial ring in $\theta$, a rational function over $\PP^1$ which has a simple pole at $\infty$ and which is regular everywhere else on $\PP^1$. We denote by $K=\Fq(\theta)$ the fraction field of $A$ and by $K_\infty=\Fq((1/\theta))$ the completion of $K$ at the place associated with $\infty$. 

Recall that $\mathbb N$ denotes the set of positive integers. In \cite{Car35} Carlitz introduced, for $n \in \mathbb N$, the elements $\zeta_A(n)$ given by the convergent series 
	\[ \zeta_A(n) := \sum_{a \in A_+} \frac{1}{a^n} \in K_\infty \]
which, through the previously mentioned analogy, correspond to classical special zeta values in the function field setting and are called the Carlitz zeta values. For any tuple of positive integers $\frak s=(s_1,\ldots,s_r) \in \mathbb N^r$, Thakur \cite{Tha04} considered the elements $\zeta_A(\frak s)$ (or $\zeta_A(s_1,\ldots,s_r)$) of $K_\infty$ defined by the convergent series 
\begin{equation} \label{definition MZV}
\zeta_A(\frak s):=\sum \frac{1}{a_1^{s_1} \ldots a_r^{s_r}} \in K_\infty
\end{equation}
where the sum runs through the set of tuples $(a_1,\ldots,a_r) \in A_+^r$ with $\deg a_1>\ldots>\deg a_r$. 
These are {\em Thakur's multiple zeta values}, characteristic $p$ analogues of the classical multiple zeta values. 
We call $r$ the {\em depth} of $\zeta_A(\frak s)$ and $w(\frak s)=s_1+\dots+s_r$ the {\em weight} of $\zeta_A(\frak s)$. We extend the definition of $\zeta_A(\frak s)$ to the empty tuple by defining the associated zeta value $\zeta_A(\emptyset)$ to be 1. We note that Carlitz zeta values $\zeta_A(n)$ are exactly depth one multiple zeta values. It can be proved that for any $\mathfrak{s}$ as above, $\zeta_A(\frak s)\in K_\infty^\times$ (see \cite{Tha09b}). A much stronger property holds: Thakur's multiple zeta values $\zeta_A(\mathfrak{s})$ and $1$ span an $\FF_p$-vector space which is also an $\FF_p$-algebra, graded by the weights (see \cite{Cha14,Tha17}).

In 1935 Carlitz \cite{Car35} introduced  analogues $BC_n\in K^\times$ ($n\geq1$) of the Bernoulli numbers of even order 
and proved an analogue of Euler's result \eqref{Euler} (see also \cite[\S 9.2]{Gos96}):
	\[ \frac{\zeta_A(n)}{\widetilde{\pi}^n}=\frac{BC_n}{\Gamma_n} \quad \text{ for all } n\geq 1, \ n\equiv 0\pmod{q-1}. \]
Here $\Gamma_n \in A$ is the {\em $n$-th Carlitz factorial}, and $\widetilde{\pi} \in \mathbb C_\infty^\times$ denotes the Carlitz period which is defined up to multiplication by an element of $\FF_q^\times$  (see \cite{Gos96,Tha04}).

Thakur called $\zeta_A(\frak s)$ {\it Eulerian} (see \cite[Definition 5.10.8]{Tha04}) if $\zeta_A(\frak s)/\widetilde{\pi}^{w(\frak s)}$ belongs to $K$.  Since $\zeta(p s_1,\dots,p s_r)=\zeta(s_1,\dots,s_r)^p$, we restrict ourselves to {\em primitive} tuples $\mathfrak{s}$, characterized by the property that not all the coefficients $s_i$ are divisible by $p$.  

Chen \cite{Che17}, Lara Rodr\'iguez and Thakur \cite{LRT14} proved the following remarkable theorem:
\begin{theorem}[Chen, Lara Rodr\'iguez and Thakur] \label{Eulerian genus 0}
Assume that $A=\Fq[\theta]$. Then for all $n\geq1$ and $k\geq 0$,  the multiple zeta value 
	\[ \zeta_A\left(q^n-1,(q-1)q^n,\ldots,(q-1)q^{n+k-1}\right) \] 
is Eulerian. 
\end{theorem}

Contrary to the classical setting, for $q>2$, it is conjectured that these series exhaust {\it all} the (primitive) Eulerian tuples when the depth exceeds $2$ (see \cite[\S 6.2]{CPY19} and \cite[\S 7.3]{Tha17}). The depth is formally defined in \S \ref{Thakur MZV}.

\subsection*{The higher genus case}

The analogy mentioned in the beginning of \S \ref{char p MZV} focuses on the comparison between the arithmetic theory over $\ZZ$ and that over an $\FF_q$-algebra $A$ as in (\ref{A}). For this reason, Carlitz's theory only gives an incomplete view.
We now move to include the general situation of the higher genus cases and present our main results. 

The definition \eqref{definition MZV} of $\zeta_A(\frak s)$ extends without modification to the case of general $A$ with $\infty\in X(\FF_q)$ such that $d_\infty=\deg(\infty)=1$ as in (\ref{A}) (see also \S \ref{Thakur MZV} for precise definitions). Following Lara Rodr\'iguez and Thakur in \cite{LRT20}, we say that $\zeta_A(\frak s)$ is {\it Eulerian} if $\zeta_A(\frak s)/\widetilde{\pi}_A^{w(\frak s)}$ is algebraic over $K$, where $\widetilde \pi_A:=\widetilde \pi_\phi \in \bC_\infty$ is the period attached to the {\em standard sign-normalized Drinfeld module of rank one} associated to $A$ (see \S \ref{DH-modules}). 

It turns out that for all $n \geq 1$ and $n \equiv 0 \pmod{q-1}$,  $\zeta_A(n)$ is Eulerian (see \cite[\S 8.18]{Gos96}  and note the proximity with the Klingen-Siegel theorem mentioned above). However, only five other cases of Eulerian multiple zeta values were known by the recent work of Lara Rodr\'iguez and Thakur \cite{LRT20}: they are all of the form $\zeta_A\left(q-1,q(q-1)\right)$ or $\zeta_A\left(q^2-1,(q-1)q^2\right)$ for one of the four base rings $A$ of class number one.  The proofs are based on explicit log-algebraicity formulas developed in \cite{Tha92,Tha09} (see for examples \cite{And94,And96,ANDTR17,ANDTR20, Pap} for further developments about log-algebraicity identities). Based on extensive use of computer aided numerical experiments, Lara Rodr\'iguez and Thakur conjectured that the appropriate generalization of the statement of Theorem \ref{Eulerian genus 0} remains valid for these four base rings (see \cite[Conjecture 3.1]{LRT20}).

\subsection{Main results} ${}$\par

In this paper we work with the general case where $A$ is an arbitrary base ring as given in \S \ref{char p MZV}. We first establish a non-commutative analogue of Zagier's formula \eqref{Zagier} in our settings, with $A$ as in (\ref{A}) and $d_\infty=1$.
%The special case of $A=\FF_q[\theta]$ is considered in \cite[Proposition 4.4.9]{Pel21}.

\begin{theorem}[Theorem \ref{Zagier formula}] \label{Theorem 1}
Let $A$ be an arbitrary base ring and let $\phi$ be the corresponding standard sign-normalized Drinfeld module of rank one. Then
\begin{equation*}
\sum_{k\geq 0}(-1)^k\zeta_A\left(q-1,\ldots,(q-1)q^{k-1}\right)\tau^k=\widetilde{\pi}_\phi^{-1}\exp_\phi \widetilde \pi_\phi\in K_\infty[[\tau]].
\end{equation*}
In particular, for all $k \in \mathbb N$,  $\zeta_A(q-1,q(q-1),\dots,q^{k-1}(q-1))$ is Eulerian.
\end{theorem} 
Here $\widetilde \pi_\phi \in \bC_\infty$ is any fundamental period attached to $\phi$ and $\exp_\phi\in K_\infty[[\tau]]$ is the exponential function associated to $\phi$ (see \S \ref{DH-modules}). As $\widetilde \pi_\phi$ is defined up to multiplication by an element of $\FF_q^\times$, the formula does not depend on its choice.

We then extend Theorem \ref{Eulerian genus 0} with explicit formulas for a base ring $A$ as in the hypotheses of Theorem \ref{Theorem 1}, and prove the aforementioned conjecture of Lara Rodr\'iguez and Thakur in full generality. 

\begin{theorem}[Theorem \ref{main theorem}] \label{Theorem 2}
Let $A$ be an arbitrary base ring and let $n \geq 1$ be an integer. We denote by $H\subset K_\infty$ the Hilbert class field of $A$. Then for all $k\geq 0$, there exists an explicitly computable constant $\alpha_{n,k}\in H$ (see \eqref{formula-for-alpha}) such that
\begin{equation*}
\zeta_A\left(q^n-1,(q-1)q^n,\ldots,(q-1)q^{n+k-1}\right)=\alpha_{n,k} \zeta_A\left(q^{n+k}-1\right).
\end{equation*}
In particular, $\zeta_A\left(q^n-1,(q-1)q^n,\ldots,(q-1)q^{n+k-1}\right)$ is Eulerian.
\end{theorem}

As pointed out by Lara Rodr\'iguez and Thakur in \cite{LRT20}, such identities also hold, in the genus zero case, at `finite levels' (that is, at the level of partial sums). However, different phenomena occur in the case of $X$ of genus $g>0$. In this case, identities of multiple zeta values (`infinite level' identities) are regulated in a different way and Theorems \ref{Theorem 1} and \ref{Theorem 2} do not seem to mirror identities at the finite level. It is in fact the analysis of the interesting phenomenology described in \cite{LRT20} that brought us to the path that gave rise to the present paper. The reader will also notice similarities in the processes behind the proofs of, on one side, \cite[Theorem 3.2]{LRT14} and, on the other side, the proof of \cite[Lemma 6.11]{GP21}. 

We now briefly describe the ingredients of the proofs. Our approach is completely different from that of the previous works. To prove Theorem \ref{Theorem 1} we provide a non-commutative factorization of the so-called Carlitz-Hayes polynomials. Then, passing to the limit, we obtain a non-commutative factorization of the function $\exp_A:=\widetilde \pi_\phi^{-1} \exp_\phi \widetilde \pi_\phi$ from which Theorem \ref{Theorem 1} follows.

The proof of Theorem \ref{Theorem 2} is based on the key notion of {\em $L$-series in Tate algebras} introduced by the third author in \cite{Pel12}. This notion has turned out to be very fruitful in function field arithmetic, see for example \cite{ANDTR17,ANDTR19,ANDTR20,AP15,APTR16,APTR18,ATR17,GP21,GP18}. In our setting, these functions can be identified with rigid analytic functions over the analytification $(X\setminus\{\infty\})^{an}_{\CC_\infty}$ of the affine curve $\operatorname{Spec}(A\otimes_{\FF_q}\CC_\infty)$ over $\CC_\infty$. Note that other classical functions of the theory, such as the exponential functions $\exp_\phi$ and $\exp_A$, are rather rigid analytic functions over the rigid affine line $\mathbb{A}^{1,an}_{\CC_\infty}$. 

%Also, the {\em special functions} of \cite{ANDTR17b} can be realized as ratios of functions of this kind, over the analytification of $\operatorname{Spec}(A)$. These functions are presented in \S \ref{MZV-Tate-alg}.

We prove a key and unexpected new general identity in Theorem \ref{key identity} among $L$-series in Tate algebras at finite levels. As a consequence, we obtain a harmonic product formula in Proposition \ref{prop-from-Z0-to-Zk} which allows us to conclude. 

Finally, as an application, for class number one base rings, we show in Theorem \ref{LRT conjecture} that our explicit expressions match with those conjectured by Lara Rodr\'iguez and Thakur \cite[Conjecture 3.3]{LRT20}.

%\subsection*{Acknowledgments.}  
%The second and third authors (T. Ngo Dac and F. Pellarin) were partially supported by the Labex MILYON ANR-10-LABX-0070.  T. Ngo Dac was also partially supported by the ANR Grant COLOSS ANR-19-CE40-0015-02. 

%%%%%%%%%%%

\subsection*{Acknowledgements} The authors are deeply thankful to Dinesh Thakur for a careful reading of a preliminary version of the present manuscript, and for useful corrections, suggestions and discussions.

\section{Background}

\subsection{Notation}

We keep the notation of the Introduction. Recall that $\mathbb N$ is the set of positive integers. We denote by $\mathbb Z^{\geq 0}$ the set of non-negative integers. Let $K$ be the function field of a geometrically connected smooth projective curve $X$ over $\Fq$ and $\infty$ be a place of $K$ of degree $d_\infty=1$. Let $g$ be the genus of $X$ and $A$ be the {\em base ring} of elements of $K$ which are regular outside $\infty$. Let $v_\infty$ be the additive valuation at $\infty$ and let us set $\deg(\cdot):=-v_\infty(\cdot)$. We denote by $K_\infty$ the completion of $K$ at $\infty$ and by $\bC_\infty$ the completion of an algebraic closure $\overline K_\infty$. Let $\overline K$ be the algebraic closure of $K$ in $\overline K_\infty$. 

We fix a uniformizer $\pi$ of $K_\infty$ so that we can identify $K_\infty$ and the field of formal Laurent series 
$\FF_q((\pi))$. This choice also allows to associate to it a sign function $\sgn:K_\infty^\times \to \Fq^\times$ which is the unique group homomorphism such that $\epsilon(\pi)=1$ (see \cite[\S 7.2]{Gos96}). We define $A_+$ to be the set of monic elements of $A$, i.e., the set of $a \in A$ such that $\sgn(a)=1$. For all $i \in \mathbb Z^{\geq 0}$, we also set $A_{+,i}=\{a \in A_+: \deg(a)=i\}$, and
\begin{align*}
A(\leq i)=\{a \in A: \deg(a) \leq i\}, \quad A(<i)=\{a \in A: \deg(a) < i\}.
\end{align*}

Further, for any field extension $L$ of $\Fq$, we have a uniquely determined sign function
	\[ \widetilde \sgn:(L \otimes_{\Fq} K_\infty)^\times \to L^\times \]
extending $\sgn$. We also set
	\[ X_L:=X \times_{\Fq} L. \]
	
\subsection{Multiple zeta values of Thakur, and in Tate algebras} \label{Thakur MZV}
We follow Thakur in \cite[\S 3]{Tha17} and Lara Rodr\'iguez-Thakur in \cite[\S 2]{LRT20}. To the datum
$(K,\infty,\sgn)$ we associate certain multiple zeta values
$$\zeta_A\big(s_1,\ldots,s_r\big)\in K_\infty$$ in the following way.
For $d\geq 0$ and $n\in\mathbb N$ we consider the {\em power sum of degree $d$ and order $n$} given by
$$S_d(n)=\sum_{a\in A_{+,d}}\frac{1}{a^n}\in K.$$ 
By convention we define empty sums to be 0 and empty products to be 1.  Thus if $A_{+,d}=\emptyset$, then $S_d(n)=0$. Note that some such sums can vanish. 
Along with this, we also set
$$S_{<d}(n):=\sum_{i=0}^{d-1}S_i(n)\in K$$
and iteratively for $r>1$ and $s_1,\ldots,s_r\in\mathbb N$,
\begin{align*}
S_d(s_1,\ldots,s_r)&:=S_d(s_1)S_{<d}(s_2,\ldots,s_r),\\ 
S_{<d}(s_1,\ldots,s_r)&:=\sum_{i=0}^{d-1}S_{i}(s_1,\ldots,s_r)\in K.
\end{align*}
It is easily seen that $S_d(n)$ tends to $0$ in $K_\infty$, as $d$ tends to $\infty$, the sequence 
$S_{<d}(s_1,\ldots,s_r)$ converges in $K_\infty$ for any choice of $s_1,\ldots,s_r\in\mathbb N$. We write
\begin{equation}\label{def-general-zeta}
\zeta_A\big(s_1,\ldots,s_r\big):=\lim_{d\rightarrow\infty}S_{<d}(s_1,\ldots,s_r)\in K_\infty.
\end{equation}
For notational convenience, we also set
$$\zeta_A(\emptyset):=1\in K_\infty.$$

For an $r$-tuple $\frak s=(s_1,\dots,s_r)\in \mathbb N^r$ we shall write $\zeta_A(\mathfrak{s})$ instead of using the long expression (\ref{def-general-zeta}). The {\em weight} of $\mathfrak{s}$ is $w(\mathfrak s)=s_1+\dots+s_r$ and its {\em depth} is
$r$. We shall also say that the presentation $\zeta_A(\frak{s})$ is of {\em depth $r$} and {\em weight}
$w(\frak{s})$. 

These elements of $K_\infty$ are the $\infty$-adic multiple zeta values of the papers \cite{LRT20,Tha17,Tha20}. At this level of generality, not all the basic properties of the multiple zeta values of
\S \ref{in-genus-0} are known to be true. For instance, it is unclear if all the values defined in (\ref{def-general-zeta}) are
non-zero. Moreover, while it is fairly well known that the $\FF_p$-subvector space they span in $K_\infty$ also is an $\FF_p$-algebra (with the product rules being identical to those in the genus zero case), it seems that we miss
a proof that the algebra they generate is graded by the weights.

\subsubsection{A family of rigid analytic functions} \label{MZV-Tate-alg} 
We recall that $\CC_\infty$ is the completion of an algebraic closure $\overline{K}_\infty$ of $K_\infty$.
We present here the basic 
{\em multiple zeta values in Tate algebras} that we study in the paper.
We consider the $\FF_q$-algebra homomorphism 
$$\chi:A\rightarrow A \otimes_{\FF_q} \CC_\infty$$
determined by $a\mapsto a \otimes 1$. There is a natural bijection $x\mapsto \chi_x$ between closed points $x\in 
X(\CC_\infty)\setminus\{\infty\}$ (where for $F$ a field extension of $\FF_q$, $X(F)$ denotes the set of $F$-rational points of $X$) and $\FF_q$-algebra homomorphisms $A\xrightarrow{\chi_x}\CC_\infty$.
We denote by $\Xi$ the unique point of $X(\CC_\infty)\setminus\{\infty\}$ such that $\chi_\Xi$ is the natural inclusion $A\subset\CC_\infty$.
For any given $a\in A$ the assignment $x\mapsto \chi_x(a)\in\CC_\infty$
defines a rational map
$$X(\CC_\infty)\setminus\{\infty\}\xrightarrow{\chi(a)}\CC_\infty$$
which is regular everywhere away from $\infty$, with a pole of order $\deg(a)$ at $\infty$.
If $x\in X(\CC_\infty)\setminus\{\infty\}$ we denote by $\chi_x(a)$, or $\chi(a)_x$, or $\chi(a)|_x$, the image of $x$ by $\chi(a)$ and we call it the {\em evaluation of $\chi(a)$ at $x$}.

If $F$ is a subfield of $\CC_\infty$, we consider the 
Tate algebra 
	$$\TT(F):=\widehat{A \otimes_{\Fq} F}_{\|\cdot\|}$$ 
with the completion $\widehat{(\cdot)}$ relative to $\|\cdot\|$ the multiplicative Gauss valuation extending a fixed multiplicative valuation $|\cdot|$ over $1 \otimes F$, trivial over $A \otimes 1$.
Let $\mathcal{B}=(\beta_i)_{i\in \mathcal{I}}$ be a basis of the $\FF_q$-vector space $A$. Note that $\mathcal I$ is countable set. Every element $f$ of $\TT(F)$ can be expanded in a unique way as a series
	$$f=\sum_{i\in \mathcal I}\beta_i f_i,\quad f_i\in \widehat{F},$$
with $\widehat{F}$ the completion of $F$, where $f_i\rightarrow0$ for the Fr\'echet filter over $\mathcal I$. Then $\|f\|=\sup_{i \in \mathcal I} |f_i|$. An {\em entire function} $f$
over $X(\CC_\infty)\setminus\{\infty\}$ is by definition one such formal series in $\TT(\CC_\infty)$ such that additionally,
$|f_i|\rho^{\deg(\beta_i)}\rightarrow0$ for all $\rho>1$. Equivalently, the sequence $(f_i\chi_x(\beta_i))_i$
tends to zero for all $x\in X(\CC_\infty)\setminus\{\infty\}$. 
We denote by $f(x)$ or $f|_x$ the sum of the convergent series 
$\sum_if_i\chi_x(\beta_i)\in\CC_\infty$. This is the {\em evaluation of the entire function $f$ at $x\in X(\CC_\infty)\setminus\{\infty\}$}.
It is not difficult to see that this evaluation does not depend on the choice of the basis $\mathcal{B}$, and that the evaluations taken with respect to different bases agree. Also, let $X^{an}_{\CC_\infty}$ be the analytification of $X$ over $\CC_\infty$. It is easy to verify that if $f=\sum_{i\in\mathcal{I}} \beta_i f_i$ is an entire function over $X(\CC_\infty)\setminus\{\infty\}$,
then the everywhere convergent series defines a rigid analytic function
$X^{an}_{\CC_\infty}\setminus\{\infty\}\rightarrow\CC_\infty$. In the case $X=\PP^1$ with $\infty$ an $\FF_q$-rational point, this 
coincides with the usual notion of an entire function $f:\CC_\infty\rightarrow\CC_\infty$ (see \cite[Chapter 3]{Gos96}). We denote by $\EE$ the 
$\CC_\infty$-algebra of entire functions $X(\CC_\infty)\setminus\{\infty\}\rightarrow\CC_\infty$.

The unique continuous and open $A \otimes 1$-linear automorphism $\tau$ of $\TT(\CC_\infty)$ extending $c\mapsto c^q$ over $1 \otimes \CC_\infty$ induces an automorphism of $\EE$. Explicitly,  if
$$f=\sum_{i\in\mathcal{I}}\beta_i f_i$$
is the series expansion of the entire function $f$, relative to the basis $\mathcal{B}$, then
$$f^{(k)}:=\tau^k(f)=\sum_{i\in\mathcal{I}}\beta_i f_i^{q^k},\quad k\in\ZZ,$$
is another entire function over $X(\CC_\infty)\setminus\{\infty\}$, well defined independently of the choice of the basis $\mathcal{B}$. If $x\mapsto x^{(1)}$ denotes the Frobenius map $X(\CC_\infty)\rightarrow X(\CC_\infty)$
then $f^{(1)}(x^{(1)})=f(x)^q$ for all $x\in X(\CC_\infty)\setminus\{\infty\}$. 

We consider the rational functions $X(\CC_\infty)\setminus\{\infty\}\rightarrow\CC_\infty$
\begin{equation} \label{eq:LTate}
\mathcal{S}_d:=\sum_{a\in A_{+,d}}\frac{\chi(a)}{a} \in A \otimes_{\Fq} K.
\end{equation}
Note that they are also entire in the above terminology.

In the present paper, we study the functions $\tau^n(Z_k)$ where, for all $k\geq 0$, we define $Z_k$ by the formal series
$$Z_k:=\sum_{d\geq 0}\mathcal{S}_dS_{<d}\Big(q-1,(q-1)q,\ldots,(q-1)q^{k-1}\Big)$$
which converges in $\TT(K)$. They are all entire functions, as shown by the following lemma:

\begin{lemma}\label{lemma-entire}
For all $k\geq 0$, the series $Z_k\in\TT(K)$ defines an entire function $X(\CC_\infty)\setminus\{\infty\}\rightarrow\CC_\infty$.
\end{lemma}

\begin{proof}
It suffices to show that for all $x\in X(\CC_\infty)\setminus\{\infty\}$, the sequence of evaluations
$$\big(\mathcal{S}_i|_x\big)_{i\geq0}$$ of $\mathcal{S}_i$ at $x$ tends to zero uniformly for the standard admissible covering of $X^{an}_{\CC_\infty}\setminus\{\infty\}$.

The proof we give is a simple adaptation of the proof of \cite[Lemma 7]{AP14}.
We set, for $r\geq 0$, $y_r=q^{r+1}-1$ which $p$-adically tends to $-1\in\ZZ_p$ as $r \to \infty$. Note that for all $r\geq 0$, assuming that $A_{+,i}\neq\emptyset$,
$$\sum_{a\in A_{+,i}}\chi(a)a^{y_r}=\sum_{b\in A(<i)}L(\eta_i+b)L_0(\eta_i+b)^{q-1}\cdots L_r(\eta_i+b)^{q-1}\in A\otimes_{\FF_q}A,$$
where $\eta_i$ is any element of $A_{+,i}$, $L$ is the $\FF_q$-linear form $a\mapsto a \otimes 1$
and $L_j$ is the $\FF_q$-linear form $a\mapsto 1 \otimes a^{q^j}$ for $j=0,\ldots,r$. By \cite[Lemma 8.8.1]{Gos96}, $\sum_{a\in A_{+,i}}\chi(a)a^{y_r}$ vanishes if $1+(r+1)(q-1)<\dim_{\FF_q}(A(<i))$. We can therefore choose
\begin{equation*}
r=\dim_{\FF_q}(A(<i))-\begin{cases}
2 \quad \text{if } q>2, \\
3 \quad \text{otherwise.}
\end{cases}
\end{equation*}
assuming at once $i$ large enough so that $r$ is non-negative. In particular we have, for this choice of $r$, the identity $\sum_{a\in A_{+,i}}\chi(a) a^{y_r}=0$.

We consider $a\in A_{+,i}$. Since $\epsilon(a)=1$, there exists a unique element $\langle a\rangle\in K_\infty^\times\cap\operatorname{Ker}(\epsilon)$, depending on $\pi$, such that $a=\langle a\rangle\pi^{-i}$. Note that $|\langle a\rangle -1|\leq |\pi|<1$. We have
$$\langle a\rangle^{-1}-\langle a \rangle^{y_r}=\sum_{j\geq 0}\Big(\binom{-1}{j}-\binom{y_r}{j}\Big)(\langle a\rangle-1)^j\in\FF_q[[\pi]].$$ Since $\binom{-1}{j}-\binom{y_r}{j}=0$ for $j=0,\ldots,q^{r+1}-1$, the above difference has multiplicative valuation strictly smaller than $|\pi|^{q^{r+1}}$. This implies that 
$\mathcal{S}_{i}\in A \otimes_{\FF_q} K_\infty$ has the property that there is a finite expansion
$\mathcal{S}_{i}=\sum_{a \in A_{+,i}} a \otimes \alpha_a$ with, for $a \in A_{+,i}$, $\alpha_a \in K_\infty$ such that $|\alpha_a|<|\pi|^{i+q^{r+1}}$. From this we deduce that 
the sequence of entire functions $\mathcal{S}_i$ tends to zero uniformly for the standard admissible covering of $X^{an}_{\CC_\infty}\setminus\{\infty\}$. The covering is given by the filtered union
$$\bigcup_{\begin{smallmatrix}\rho\in|\CC_\infty^\times| \\ \rho>1\end{smallmatrix}}\operatorname{Spm}\Big(\widehat{A \otimes \CC_\infty}_{\|\cdot\|_\rho}\Big),$$
where $\operatorname{Spm}$ denotes the maximal spectrum, and where we take the completions with respect to $\|\cdot\|_\rho$ the multiplicative valuation which restricts to $|\cdot|$ over $1 \otimes \CC_\infty$, 
such that $\|a\otimes 1\|_\rho=\rho^{\deg(a)}$. By the above discussion, for all $x\in 
\operatorname{Spm}\big(\widehat{A \otimes \CC_\infty}_{\|\cdot\|_\rho}\big)$ we have, by the maximum principle, and for large values of $i$:
$$|\mathcal{S}_i(x)|\leq \|\mathcal{S}_i\|_\rho\leq \rho^i |\pi|^{i+q^{r+1}}$$ which tends to zero for any fixed
$\rho>1$ as $i\rightarrow\infty$.\end{proof}

\begin{remark}
The entire function $Z_k$ can be viewed as an example of multiple zeta value in the Tate algebra $\TT(K)$. 
These functions have been investigated in 
\cite{GP21} in the case of $X=\PP^1$ and $\infty$ an $\FF_q$-rational point. We will not pursue the
general theory of such functions because in the present paper we only need to work with the entire functions
$$\tau^n(Z_k),\quad n,k\geq 0.$$
However, we mention here that, in agreement with the notations of ibid.,
we could have written:
$$Z_k=\zeta_A\begin{pmatrix}\{1\} & \emptyset & \cdots & \emptyset \\ 1 & q-1 & \cdots & (q-1)q^{k-1}\end{pmatrix}\in \TT(K),$$
to stress the analogy. Note that, following this notation, we have $Z_0=\zeta_A\binom{\{1\}}{1}$. We also set $Z_{-1}:=0$. Also, Lemma \ref{lemma-entire} can be easily generalized to match with the formalism of 
\cite{GP21} and the reader can see
that all the functions that naturally arise in this picture define entire functions over products of the curve
$X(\CC_\infty)\setminus\{\infty\}$ in an appropriate sense.
\end{remark}

%%%%%%%%%%%%%%%%%%%%

\section{A universal formula}

\subsection{Auxiliary results}

We recall several results from \cite[\S 0 and \S 1]{Tha92}. For all $i\in\ZZ^{\geq0}$ and $n \in \mathbb N$, we set, with appropriate conventions for empty sums and products, 
\begin{align*}
D_i &:=\prod_{a \in A_{+,i}} a, \\
e_i(z) &:=\prod_{a \in A(<i)} (z-a), \\
\frac{e_i(z)}{D_i} &:= \sum_{k \geq 0} A_{ik} z^{q^k} \\
S_i(n) &:= \sum_{a \in A_{+,i}} \frac{1}{a^n}, \\
S_i & :=S_i(1) =\sum_{a \in A_{+,i}} \frac{1}{a}.
\end{align*}
If $A_{+,i} \neq \emptyset$, then by \cite[Eq. (13)]{Tha92},
	\[ e_{i+1}(z)=e_i(z)^q-D_i^{q-1} e_i(z). \]
Otherwise, if $A_{+,i} = \emptyset$, we get $e_{i+1}(z)=e_i(z)$.
In particular, for all $i \in \ZZ^{\geq0}$, $D_{i+1} A_{(i+1)0}$ and $D_i^q A_{i0}$ agree up to a sign so that we can write $D_{i+1} A_{(i+1)0}=\pm D_i^q A_{i0}$ and 
	\[ A_{i0}=\pm \frac{(D_0 D_1 \dots D_{i-1})^{q-1}}{D_i} \neq 0. \]

Now, if $A_{+,i} \neq \emptyset$, by \cite[Eq. (18)]{Tha92}, we have
	\[ \frac{A_{i0}}{1-\sum_{k=0}^i A_{ik} z^{q^k}}=-\sum_{a \in A_{+,i}} \frac{1}{z-a}=\sum_{n \geq 0} S_i(n+1) z^n. \]
Consequently, 
\begin{equation} \label{Si nonzero}
S_i=S_i(1)=A_{i0} \neq 0,
\end{equation}
and for all $1 \leq n \leq q-1$, 
\begin{equation} \label{Si(n)}
S_i(n) = S_i^n.
\end{equation}

For all $d \in \mathbb Z^{\geq 0}$, we recall that we have set (see \eqref{eq:LTate})
\begin{align*}
\mathcal S_d &=\sum_{a \in A_{+,d}} \frac{\chi(a)}{a}.
\end{align*}
We additionally set 
\begin{align*}
\mathcal S_{<d} :=\sum_{i=0}^{d-1}\mathcal S_i,\quad 
S_{<d} :=\sum_{i=0}^{d-1} S_i.
\end{align*}
We put $\mathcal S_{\leq d}= \mathcal S_{<d+1}$ and $S_{\leq d}=S_{<d+1}$.

If $D$ is a divisor over $X_{\overline{K}}$, then we denote by $\mathcal{L}(D)$ the $\overline{K}$-subvector space of the fraction field of $X_{\overline K}$ consisting of functions $f$ such that $(f)+D$ is effective: $(f)\geq -D$. For all $m \in \mathbb Z^{\geq 0}$, we define $j_m$ to be the smallest non-negative integer such that $\dim_{\overline K} \mathcal{L}(j_m\infty)=m+1$. Note that $\mathcal{L}(j_m (\infty))=A(\leq j_m)\otimes_{\FF_q}\overline{K}$,
where $A(\leq \ell)$ is the $\FF_q$-vector space spanned by the elements $a \in A$ such that $\deg(a)\leq \ell$. We have $j_0=0<j_1<j_2<...$ and for $m$ sufficiently large, $j_m=m+g$ by the Riemann-Roch theorem. 

\begin{lemma} \label{zeros}
For all $m\geq 1$, the function $\mathcal S_{j_m}$ vanishes at $\Xi,\dots,\Xi^{(m-1)}$.
\end{lemma}

\begin{proof}
Let us choose an element $\eta_m\in A_{+,j_m}$.
Then $A_{+,j_m}=\eta_m+A(\leq j_{m-1})$. Note that by the definition of the sequence $\{j_m\}_{m \geq 0}$, $\dim_{\FF_q} A(\leq j_{m-1})=m$. Since
$$\mathcal{S}_{j_m}(\Xi^{(j)})=\sum_{a \in A_{+,j_m}} a^{q^j-1}=\sum_{b \in A(\leq j_{m-1})} L_0(\eta_m+b)^{q-1}\cdots L_{j-1}(\eta_m+b)^{q-1}$$ where $L_i$ is the $\FF_q$-linear form $a\mapsto a^{q^i}$ for $i=0,\ldots,j-1$, $\mathcal{S}_{j_m}(\Xi^{(j)})$ vanishes for $j=0,\ldots,m-1$ in virtue of \cite[Lemma 8.8.1]{Gos96}.
\end{proof}

\begin{lemma} \label{basis}
The set $\{\mathcal S_{j_0},\dots,\mathcal S_{j_m}\}$ forms a basis of $\mathcal L(j_m \infty)$.
\end{lemma}

\begin{proof}
We observe that for all $k \in \mathbb Z^{\geq 0}$, $\mathcal S_{j_k}$ has a pole of order $j_k$ at $\infty$. The integers $j_0,\dots,j_m$ being distinct, $\mathcal S_{j_0},\dots,\mathcal S_{j_m}$ are linearly independent over $\overline K$. Since $\dim_{\overline K} \mathcal L(j_m \infty)=m+1$, the Lemma follows.
\end{proof}

\begin{lemma} \label{nonvanishing}
The function $\mathcal S_{j_m}$ does not vanish at $\Xi^{(m)}$.
\end{lemma}

\begin{proof}
We suppose by contradiction that $\mathcal S_{j_m}$ vanishes at $\Xi^{(m)}$. Combining with Lemma \ref{zeros} implies that for all $k \geq m$, $\mathcal S_{j_m},\dots,\mathcal S_{j_k}$ belong to $\mathcal L(j_k \infty -\Xi-\dots-\Xi^{(m)})$. We choose $k$ sufficiently large so that by the Riemann-Roch theorem, 
\begin{itemize}
\item $\dim_{\overline K} \mathcal L(j_k \infty )=j_k+1-g$,
\item $\dim_{\overline K} \mathcal L(j_k \infty - \Xi-\dots-\Xi^{(m)})=j_k-(m+1)+1-g$.
\end{itemize}
The first equation implies $j_k+1-g=k+1$. Putting this into the second equation yields 
	\[ \dim_{\overline K} \mathcal L(j_k \infty- \Xi-\dots-\Xi^{(m)})=j_k-(m+1)+1-g=k-m. \] 
By Lemma \ref{basis}, we know that $\mathcal S_{j_m},\dots,\mathcal S_{j_k}$ are linearly independent over $\overline K$. Thus we get a contradiction since 
\begin{align*}
k-m &=\dim_{\overline K} \mathcal L(j_k \infty - \Xi-\dots-\Xi^{(m)}) \\\
&\geq \dim_{\overline K} \text{Vect}(\mathcal S_{j_m},\dots,\mathcal S_{j_k}) \\
&=k-m+1. 
\end{align*}
We conclude that $\mathcal S_{j_m}$ does not vanish at $\Xi^{(m)}$.
\end{proof}

\subsection{A key identity}

We now prove the main result of this section.

\begin{theorem} \label{key identity}
For all $i \in  \mathbb Z^{\geq 0}$, we have
	\[ \mathcal S_i^{(1)}=\mathcal S_{\leq i} \cdot S_i(q-1). \]
\end{theorem}

\begin{proof}
Note that if $i \neq j_m$ for all $m$, then $A_{+,i}=\emptyset$. Thus $\mathcal S_i=S_i(q-1)=0$ and we are done. Therefore, it suffices to prove that for all $m \in  \mathbb Z^{\geq 0}$, 
\begin{equation} \label{identity1}
\mathcal S_{j_m}^{(1)}=\mathcal S_{\leq j_m} \cdot S_{j_m}(q-1). 
\end{equation}
We observe 
\begin{align*}
\mathcal S_{j_m}^{(1)} \Big|_\Xi=S_{j_m}(q-1)=S_{j_m}^{q-1}.
\end{align*}
Here the second equality holds by (\ref{Si(n)}). By \eqref{Si nonzero}, it follows that $\mathcal S_{j_m}^{(1)} \Big|_\Xi \neq 0$. Thus \eqref{identity1} is equivalent to 
\begin{equation} \label{identity2}
\frac{\mathcal S_{j_m}^{(1)}}{\mathcal S_{j_m}^{(1)} \Big|_\Xi}=\mathcal S_{\leq j_m}. 
\end{equation}

We now prove \eqref{identity2} by induction on $m$. For $m=0$, \eqref{identity2} holds since $\mathcal S_{\leq j_m}=\mathcal S_{j_m}^{(1)}=S_{j_m}(q-1)=1$. Suppose that for all $0 \leq k \leq m-1$,
	\[ \frac{\mathcal S_{j_k}^{(1)}}{\mathcal S_{j_k}^{(1)} \Big|_\Xi}=\mathcal S_{\leq j_k}. \]
We show that 
	\[ \frac{\mathcal S_{j_m}^{(1)}}{\mathcal S_{j_m}^{(1)} \Big|_\Xi}=\mathcal S_{\leq j_m}. \]

In fact, 	
\[ F:=\frac{\mathcal S_{j_m}^{(1)}}{\mathcal S_{j_m}^{(1)} \Big|_\Xi} \in \mathcal L(j_m \infty). \]
By Lemma \ref{basis}, the set $\mathcal S_{j_0},\dots,\mathcal S_{j_m}$ forms a basis of the $\overline K$-vector space $\mathcal L(j_m \infty)$. Therefore, there exist $a_0,\dots,a_m \in \overline K$ such that 
\begin{equation} \label{expression}
F= a_0 \mathcal S_{j_0}+\dots+a_m \mathcal S_{j_m}.
\end{equation}

We show by induction on $0 \leq k \leq m-1$ that $a_k=1$. For $k=0$, we evaluate \eqref{expression} at $\Xi$. The left-hand side gives $F |_\Xi=1$. By Lemma \ref{zeros}, the right-hand side returns
	\[ a_0 \mathcal S_{j_0} \Big|_\Xi+\dots+a_m \mathcal S_{j_m} \Big|_\Xi=a_0. \]
We then obtain $a_0=1$. 

Suppose that for all $j$ such that $0 \leq j <k$, we have $a_j=1$. We have to prove that $a_k=1$. We evaluate \eqref{expression} at $\Xi^{(k)}$. Note that $1 \leq k \leq m-1$. By Lemma \ref{zeros},
	\[ F |_{\Xi^{(k)}}=0. \]
We get
\begin{align*}
(a_0 \mathcal S_{j_0}+\dots+a_m \mathcal S_{j_m}) \Big|_{\Xi^{(k)}}&=a_0 \mathcal S_{j_0} \Big|_{\Xi^{(k)}}+\dots+a_m \mathcal S_{j_m} \Big|_{\Xi^{(k)}} \\ 
&=a_0 \mathcal S_{j_0} \Big|_{\Xi^{(k)}}+\dots+a_k \mathcal S_{j_k} \Big|_{\Xi^{(k)}} \\
&=\mathcal S_{j_0} \Big|_{\Xi^{(k)}}+\dots+\mathcal S_{j_k} \Big|_{\Xi^{(k)}}+(a_k-1) \mathcal S_{j_k} \Big|_{\Xi^{(k)}} \\
&=(\mathcal S_{j_0}+\dots+\mathcal S_{j_k}) \Big|_{\Xi^{(k)}}+(a_k-1) \mathcal S_{j_k} \Big|_{\Xi^{(k)}} \\
&=\mathcal S_{\leq j_k} \Big|_{\Xi^{(k)}}+(a_k-1) \mathcal S_{j_k} \Big|_{\Xi^{(k)}}.
\end{align*}
Here the second equality follows from Lemma \ref{zeros} again. The third equality holds as $a_0=\dots=a_{k-1}=1$.

Since $k \leq m-1$, by the induction hypothesis and Lemma \ref{zeros}, 
	\[ \mathcal S_{\leq j_k} \Big|_{\Xi^{(k)}}=\frac{\mathcal S_{j_k}^{(1)} \Big|_{\Xi^{(k)}}}{{\mathcal S_{j_k}^{(1)} \Big|_\Xi}}=0. \]
We then get
\begin{align*}
(a_0 \mathcal S_{j_0}+\dots+a_m \mathcal S_{j_m}) \Big|_{\Xi^{(k)}} =(a_k-1) \mathcal S_{j_k} \Big|_{\Xi^{(k)}}.
\end{align*}
Putting all together yields
	\[ 0=F |_{\Xi^{(k)}}=a_0 \mathcal S_{j_0} \Big|_{\Xi^{(k)}}+\dots+a_m \mathcal S_{j_m} \Big|_{\Xi^{(k)}} =(a_k-1) \mathcal S_{j_k} \Big|_{\Xi^{(k)}}. \] 
By Lemma \ref{nonvanishing}, $\mathcal S_{j_k} \Big|_{\Xi^{(k)}} \neq 0$. We deduce that $a_k=1$ as desired.

To finish the proof, we show that $a_m=1$. We look at the sign of both sides of \eqref{expression}. The left-hand side gives
\begin{align*}
\widetilde \sgn(F)=\frac{\widetilde \sgn(\mathcal S_{j_m}^{(1)})}{\mathcal S_{j_m}^{(1)} \Big|_\Xi}=\frac{S_{j_m}^q}{S_{j_m}(q-1)}=S_{j_m}.
\end{align*}
Here the last equality follows from Lemma \ref{Si(n)}. The right-hand side equals $a_m S_{j_m}$. Thus $S_{j_m}=a_m S_{j_m}$. As $S_{j_m} \neq 0$, we deduce $a_m=1$ and the proof is complete.  
\end{proof}

As a direct consequence, since $Z_0=\lim_{n\rightarrow\infty}\mathcal{S}_{\leq n}\in\TT(K)$, we deduce that 
the entire function $Z_0$ vanishes at the points $\Xi^{(1)},\Xi^{(2)},\ldots$ and $Z_0\big|_\Xi=1$. The genus one case of Theorem \ref{key identity} can also be deduced from the explicit arguments of Green and Papanikolas in \cite{GP18}.

%%%%%%%%%%%%%%%%%%%%

\section{Universal families of Eulerian multiple zeta values}

%$\zeta_A\left(\begin{pmatrix} \chi_t \\ 1 \end{pmatrix},q-1,\dots,q^k(q-1)\right)$

\subsection{Sign-normalized rank one Drinfeld modules}
\label{DH-modules}
If $R$ is an $\FF_q$-algebra we denote by $R[\tau]$ the skew polynomial ring in $\tau$ (denoted by $R\{\tau\}$ in \cite[Chapter 1]{Gos96}), and by $R[[\tau]]$ the `ring of Frobenius power series' (denoted by $R\{\{\tau\}\}$ in \cite[\S 4.6]{Gos96}.)
Let $\phi: A\rightarrow \overline{K}_\infty[\tau]$ be a rank one Drinfeld module such that for all $a\in A\setminus\{0\}$,
	$$\phi_a=a+\cdots + \sgn(a) \tau^{\deg a}.$$
Such a Drinfeld module $\phi$ is said to be \textit{sign-normalized}. In \cite{GP18} a sign-normalized rank one Drinfeld module is called a Drinfeld-Hayes module. By \cite[Theorem 7.2.15]{Gos96}, there always exist sign-normalized rank one Drinfeld modules. Further, if we denote by $H$ the Hilbert class field of $A$ and $O_H$ the integral closure of $A$ in $H$, then for all $a \in A$, $\phi_a \in O_H[\tau]$. 

There exist unique elements $\exp_\phi, \log_\phi \in H[[\tau]]$ such that
 $$\exp_\phi,\log_\phi \in 1+ H[[\tau]]\tau,$$
 $$\forall a\in A, \quad \exp_\phi a= \phi_a\exp_\phi,$$
 $$\exp_\phi \log_\phi=\log_\phi \exp_\phi =1.$$
We write
\begin{align*}
\exp_\phi &=\sum_{i \geq 0} \epsilon_i \tau^i, \quad \epsilon_i \in H, \\
\log_\phi &=\sum_{i \geq 0} \lambda_i \tau^i, \quad \lambda_i \in H.
\end{align*} 
The formal series $\exp_\phi,$ and $\log_\phi$ are respectively called \textit{the exponential series} and \textit{the logarithm series} associated to $\phi$. Additionally, $\exp_\phi$ defines, in a unique way, 
an $\FF_q$-linear entire function $\CC_\infty\rightarrow\CC_\infty$ (unlike the functions $Z_k$ which are
entire $X(\CC_\infty)\setminus\{\infty\}\rightarrow\CC_\infty$). 
In what follows, we identify the operator $\exp_\phi$ with this entire function $\CC_\infty\rightarrow\CC_\infty$ hence tolerating an abuse of notation.
The kernel of $\exp_\phi$ is a projective $A$-module of rank one embedded discretely in $\CC_\infty$.
By \cite[Theorem 5.8]{Tha93}, the coefficients $\lambda_i$ are non-zero for all $i\geq 0$. Theorem 3.2 of \cite{Tha93} describes non-vanishing criteria for the coefficients $\epsilon_i$'s at a superior level of generality, i.e. with $\infty$ a closed point of $X$ not necessarily $\FF_q$-rational. Note that the condition $d_\infty=1$ ensures that they are all non-zero.

Drinfeld proved that there is a bijection between the set of sign-normalized rank one Drinfeld module $\phi$ and a set of certain functions $f$ called {\it shtuka functions} in the rational  function field $F:=\operatorname{Frac}(A \otimes_{\FF_q} \overline{K})$ of $X_{\overline{K}}$ (see \cite{Dri74,Dri77,Tha93}). It follows that one can associate a shtuka function $f_\phi\in F^\times$  to any sign-normalized Drinfeld module of rank one $\phi$. The divisor $(f_\phi)$ associated to $f_\phi$ can be expressed in the following way:
	\[ (f_\phi)=V_\phi^{(1)}-V_\phi+\Xi-\infty, \]
for some effective divisor $V_\phi$ of $X(\overline{K})$ of degree $g$. We call it {\it the Drinfeld divisor} attached to $\phi$. Note that the points $\Xi$ and $\infty$  do not belong to the support of $V_\phi$ (see \cite[Corollary 0.3.3]{Tha93}). 

There exists a unique sign-normalized Drinfeld module of rank one $\phi$, called the {\em standard} sign-normalized Drinfeld module, such that the kernel of $\exp_\phi$ is a free $A$-module of rank one. From now on, we only work with this standard Drinfeld module $\phi$. Then,
	\[ \operatorname{Ker}(\exp_\phi)=A \widetilde \pi_\phi \] 
for some element $\widetilde \pi_\phi \in \bC_\infty^\times$, uniquely defined up to multiplication by a factor in $\FF_q^\times$. We also denote by $f$ and $V$, respectively, the associated shtuka function and the associated Drinfeld divisor.

We now come back to the multiple zeta values of Thakur, in $K_\infty$. We choose an $r$-tuple $\mathfrak{s}=(s_1,\ldots,s_r)\in \mathbb N^r$, with $r>0$. We recall from the introduction the following definition (see the first paragraph concerning the higher genus cases).

\begin{definition}\label{def-eulerian}
The multiple zeta value $\zeta_A(\frak s)$ is {\it Eulerian} if $\zeta_A(\frak s)/\widetilde \pi_{\phi}^{w(\frak s)}\in\CC_\infty$ is algebraic over $K$.
\end{definition}

\begin{remark}
We consider the genus $0$ case, i.e., $X=\PP^1$ with $\infty$ an $\FF_q$-rational point. Then as already noticed, $A=\Fq[\theta]$ and $K=\Fq(\theta)$ for some rational function $\theta$ over $\mathbb P^1$. In this case we have $\phi=C$ the {\em Carlitz module}. It is uniquely defined by the identity $C_\theta=\theta+\tau$ in $K[\tau]$. We write $\widetilde{\pi}=\widetilde{\pi}_C$ and we have already mentioned that  if $q-1\mid n$ with $n \geq 1$, $\zeta_A(n)$ is Eulerian, see \cite{Gos96,Tha04} for more details.  By using a positive characteristic variant of Baker's theory and a seminal result of Anderson and Thakur \cite{AT90}, Yu proved the following result for $n>0$: $\zeta_A(n)/\widetilde \pi^{n} \in \overline K$ if and only if $\zeta_A(n)/\widetilde \pi^{n} \in K$, and of course the latter property occurs if $q-1\mid n$ (see \cite[Corollary 2.6]{Yu91}). This result was generalized for multiple zeta values associated to $\FF_q[\theta]$, of arbitrary depth, by Chang \cite{Cha14}. Chang and Yu explicitly computed in \cite{CY07,Yu97} the transcendence degree of the subfield of
$\CC_\infty$ generated by $K$ and the zeta values $\zeta_A(n)$, $n\geq 1$. The methods of \cite{Cha14, CY07} do not use Baker's theory but rest on the transcendence methods of Anderson, Brownawell and Papanikolas \cite{ABP04}, and Papanikolas \cite{Pap08}. 
\end{remark}

Recall that $\exp_\phi =\sum_{i \geq 0} \epsilon_i \tau^i$ with $\epsilon_i \in H$. Let us consider
	$$\exp_A(z):=z \prod_{a\in A \setminus \{0\}}\left(1-\frac{z}{a}\right),\quad z\in\CC_\infty.$$
This product converges for all $z\in\CC_\infty$ and defines an entire $\FF_q$-linear 
function $\CC_\infty\rightarrow\CC_\infty$ with kernel $A\subset\CC_\infty$ that we identify with 
an operator $$\exp_A=\sum_{k\geq 0}e_i\tau^i\in K_\infty[[\tau]].$$
The link between the operators $\exp_A$ and $\exp_\phi$ is given by the following identity in $\CC_\infty[[\tau]]$:
\begin{equation}\label{expA-phi}
\exp_A=\widetilde{\pi}_\phi^{-1} \exp_\phi\widetilde{\pi}_\phi.
\end{equation}
We also have the identity of meromorphic functions over $\CC_\infty$ and of formal series of $K_\infty[[z]]$:

$$\frac{z}{\exp_A(z)}=1+\sum_{\begin{smallmatrix} n\geq 1\\ (q-1)\mid n\end{smallmatrix}}\zeta_A(n)z^n\in K_\infty[[z]].$$
From it we get
\begin{equation}\label{zeta}
\zeta_A(q^k-1)=\lambda_k\widetilde{\pi}_\phi^{q^k-1},\quad k\geq 0.
\end{equation}
Note that this implies $\lambda_k\neq0$ for all $k\geq0$, which is Theorem 5.8 of \cite{Tha93} (and essentially the same proof). Further, we deduce that for all $n \geq 1$ with $n \equiv 0 \pmod{q-1}$, $\zeta_A(n)$ is Eulerian (see \cite[Lemma 8.18.1]{Gos96}).

\subsection{Carlitz-Hayes polynomials}

We know that for all $n \geq 1$ there exists a unique polynomial called {\it the $n$-th Carlitz-Hayes polynomial}
\begin{equation*} 
	\widetilde{\mathcal E}_n\in K[\tau]
\end{equation*}
such that, for all $z \in \bC_\infty$,
\begin{equation*} 
\widetilde{\mathcal E}_n(z)=z \prod_{a \in A(<n)\setminus \{0\}} \left( 1-\frac{z}{a}\right).
\end{equation*}
Here $\widetilde{\mathcal{E}}_n(z)$ denotes the evaluation of $\widetilde{\mathcal{E}}_n$ at $z$.
Let $\kappa_{n,m}$ be the coefficient of $\tau^m$ in $\mathcal{E}_n$ seen as a linear combination 
of $1,\tau,\tau^2,\ldots$
\begin{proposition}
With the above notation, we have
	\[ \kappa_{n,m}=(-1)^m S_{<n}(q-1,q(q-1),\dots,q^{m-1}(q-1))\]
\end{proposition}

\begin{proof}
Since the zeros of $\widetilde{\mathcal E}_n$ are exactly the elements of the $\FF_q$-vector space $A(<n)$ with multiplicity $1$, we can write
	\[ \widetilde{\mathcal E}_n=(1-\alpha_{n-1} \tau) \widetilde{\mathcal E}_{n-1} \]
for some $\alpha_{n-1} \in K$. Thus we obtain a non-commutative factorization 
	\[ \widetilde{\mathcal E}_n=(1-\alpha_{n-1} \tau) \cdots (1-\alpha_1 \tau)(1-\alpha_0\tau) \]
with $\alpha_0=1$. It follows that
\begin{equation} \label{kappa1}
\kappa_{n,1}=-(\alpha_0+\alpha_1+\cdots+\alpha_{n-1}).
\end{equation}

Now, by considering the zeros, degree and the constant term, we see that, for any $a\in\CC_\infty^\times$,
	\[ \prod_{c \in \Fq^\times} \left(1-\frac{z}{ca}\right)=1-\frac{z^{q-1}}{a^{q-1}}. \]
Thus
\begin{equation*} 
\widetilde{\mathcal E}_n(z)=z \prod_{\substack{a \in A(<n)\setminus \{0\}, \\ a \text{ monic}}} \left( 1-\frac{z^{q-1}}{a^{q-1}}\right).
\end{equation*}
We deduce
\begin{equation} \label{kappa2}
\kappa_{n,1}=- \sum_{\substack{a \in A(<n)\setminus \{0\}, \\ a \text{ monic}}} \frac{1}{a^{q-1}}. 
\end{equation}
Combining \eqref{kappa1} and \eqref{kappa2} yields, inductively,
	\[ \alpha_{n-1}=\sum_{a \in A_{+,n-1}} \frac{1}{a^{q-1}}=S_{n-1}(q-1). \]
We then get an explicit factorization
	\[ \widetilde{\mathcal E}_n=\Big(1-S_{n-1}(q-1) \tau\Big) \dots \Big(1-S_1(q-1) \tau\Big). \]
The Proposition follows.
\end{proof}
Note that the multiple zeta value $\zeta_A\left(q-1,\ldots,(q-1)q^{k-1}\right)$ has depth $k$. If $k=0$,  we have set
$\zeta_A\left(q-1,\ldots,(q-1)q^{k-1}\right)=\zeta_A(\emptyset)=1$.

\begin{theorem} \label{Zagier formula}
We have the following identity
\begin{equation} \label{expA}
\exp_A=\sum_{k\geq 0}(-1)^k\zeta_A\left(q-1,\ldots,(q-1)q^{k-1}\right)\tau^k\in K_\infty[[\tau]].
\end{equation}
In particular, for all $k \in \mathbb N$, 
	\begin{equation}
	\label{epsilon-n}
	\zeta_A(q-1,q(q-1),\dots,q^{k-1}(q-1))=(-1)^k \epsilon_k \widetilde \pi_\phi^{q^k-1}. \end{equation}
\end{theorem} 

\begin{proof}
We observe that $\widetilde{\mathcal E}_n(z)$ converges uniformly on every compact subset of $\bC_\infty$ to the function $\exp_A \in K_\infty[[\tau]]$. By the previous discussion, \eqref{expA} follows immediately and, viewing (\ref{expA-phi}), it suffices to justify (\ref{epsilon-n}).
\end{proof}

As a direct consequence of Theorem \ref{Zagier formula} we deduce that the multiple zeta values $\zeta_A(q-1,q(q-1),\dots,q^{k-1}(q-1))$ are all Eulerian in the sense of Definition \ref{def-eulerian}.

Note that the special case of $A=\FF_q[\theta]$ is considered in \cite[Proposition 4.4.9]{Pel21}. Unlike 
(\ref{zeta}), (\ref{epsilon-n}) does not seem to deliver non-vanishing properties of the coefficients $\epsilon_k$.
We can however consider the reverse implication and deduce that the multiple zeta values $\zeta_A(q-1,q(q-1),\dots,q^{k-1}(q-1))$ are non-zero by using \cite[Theorem 3.2]{Tha93}.
 
\subsection{Harmonic products}

We are ready to prove our main result, by variants of the harmonic products discussed in \cite{GP21}.
\begin{proposition}\label{prop-from-Z0-to-Zk}
For all $k \in \mathbb N$ we have the following identity of entire functions over $\operatorname{Spec}(A)^{an}_{\CC_\infty}$:
\begin{align*}
Z_0 \cdot \zeta_A(q-1,\dots,q^{k-1}(q-1)) = Z_k + Z_{k-1}^{(1)}.
\end{align*}
\end{proposition}

\begin{proof}
By Lemma \ref{key identity},
\begin{align*}
& Z_0\cdot \zeta_A(q-1,\dots,q^k(q-1)) - Z_k \\
&=\sum_{i \geq 0} \mathcal S_{\leq i}  \cdot S_i(q-1,q(q-1),\dots,q^k(q-1)) \\
&=\sum_{i \geq 0} \mathcal S_{\leq i} \cdot S_i(q-1) \cdot S_{<i}(q(q-1),\dots,q^k(q-1)) \\
&= \sum_{i \geq 0} \mathcal S_i^{(1)} \cdot S_{<i}(q(q-1),\dots,q^k(q-1)) \\
&= Z_{k-1}^{(1)}.
\end{align*}
\end{proof}

As a consequence, we obtain
\begin{theorem} \label{main theorem}
Let $n\geq1$ be an integer. Then for all $k\geq 0$, there exists $\alpha_{n,k}\in H$ such that
\begin{equation}\label{identity-mzv}
\zeta_A\left(q^n-1,(q-1)q^n,\ldots,(q-1)q^{k-1+n}\right)=\alpha_{n,k}\zeta_A\left(q^{n+k}-1\right).
\end{equation}
Further,
\begin{equation}\label{formula-for-alpha}
\alpha_{n,k}=(-1)^{k}\lambda_{n+k}^{-1}\sum_{i=0}^{k}\lambda_{n+i}\epsilon_{k-i}^{q^{n+i}}.
\end{equation}

In particular, the multiple zeta value $\zeta_A(q^n-1,q^n(q-1),\dots,q^{n+k-1}(q-1))$ is Eulerian.
\end{theorem} 

Note that the depth of the multiple zeta value in the left-hand side of (\ref{identity-mzv}) equals $k+1$. 
The coefficient $\alpha_{n,k}$ is well defined because $\lambda_{n+k}$ never vanishes, as previously noticed.
If $k=0$ the left-hand side is just $\zeta_A(q^n-1)$ so that $\alpha_{n,k}=1$ in this case.

\begin{proof}[Proof of Theorem \ref{main theorem}]
Let us write
	$$\mathcal{F}:=\sum_{k\geq 0}(-1)^kZ_k\tau^k\in \TT(K)[[\tau]].$$
The coefficient of $1=\tau^0$ is therefore $Z_0$.
We also set
	$$\mathcal{G}:=\sum_{k\geq 0}(-1)^k\zeta_A\left(q-1,\ldots,(q-1)q^{k-1}\right)\tau^k\in K_\infty[[\tau]].$$
We recall that by \eqref{expA}, 
	$$\mathcal{G}=\exp_A\in K_\infty[[\tau]].$$ Hence Proposition \ref{prop-from-Z0-to-Zk} implies
the identity in $\TT(K)[[\tau]]$:
	$$\mathcal{F}-\tau\mathcal{F}=Z_0\exp_A.$$
This allows to write
	$$\mathcal{F}=\sum_{i\geq 0}\tau^i(Z_0)\tau^i\exp_A,$$
a series expansion which is convergent for the $\tau$-adic topology.
Explicitly, equating the coefficients of $\tau^k$, we find
\begin{equation}\label{a-formula-for-Zk}
Z_k=(-1)^k\sum_{i=0}^k\epsilon_{k-i}^{q^i}\tau^{i}(Z_0)\widetilde{\pi}_\phi^{q^k-q^i},\quad k\geq0.
\end{equation}
Applying $\tau^n$ we get
$$\tau^n(Z_k)=(-1)^k\sum_{i=0}^k\epsilon_{k-i}^{q^{i+n}}\tau^{i+n}(Z_0)\widetilde{\pi}_\phi^{(q^k-q^i)q^n}.$$
Evaluating at $\Xi$ the above identity of entire functions and using (\ref{expA-phi}) and (\ref{zeta}) together with the identities
$$\tau^i(Z_0)_\Xi=\zeta_A(q^i-1),\quad i\geq0,$$
we obtain 
\begin{multline*}
\zeta_A\left(q^n-1,(q-1)q^n,\ldots,(q-1)q^{k-1+n}\right)=\\ =(-1)^{k}\lambda_{n+k}^{-1}\left(\sum_{i=0}^{k}\lambda_{n+i}\epsilon_{k-i}^{q^{n+i}}\right)\zeta_A(q^{k+n}-1),\end{multline*}
and (\ref{formula-for-alpha}) holds.
This completes the proof of the Theorem.
\end{proof}

\begin{remark} 
Applying $\tau^n$ for $n>0$ to both sides of the identity of Proposition \ref{prop-from-Z0-to-Zk}, evaluating at $\Xi$, 
and assuming that $X=\PP^1$, yields directly \cite[Theorem 5.1]{Che17}. If $k=1$, (\ref{a-formula-for-Zk}) yields
$$Z_k=-\epsilon_1Z_0\widetilde{\pi}^{q-1}-Z_0^{(1)}.$$
Since $\epsilon_1=-\lambda_1$ the above is equivalent, by  (\ref{zeta}), to the harmonic product formula
\begin{equation}\label{special-identity}Z_0\zeta_A(q-1)=Z_0^{(1)}+Z_1.\end{equation}
A particular case of this formula, when $\phi=C$ is Carlitz's module, so that $A=\FF_q[\theta]$, appears in \cite[\S 7.2.1]{GP21}.
\end{remark}

\begin{remark}
Applying $\tau^n$ with $n>0$ to both sides of (\ref{special-identity}) and evaluating at $\Xi$ yields the formula
	$$\zeta_A(q^n-1)\zeta_A(q-1)^{q^n}=\zeta_A(q^{n+1}-1)+\zeta_A\left(q^n-1,(q-1)q^n\right)$$
which is a generalization of \cite[Theorem 4]{Tha09}. The case $n=0$ makes sense too, but returns a tautological identity. Theorem \ref{main theorem} in the case $k=1$ rewrites as
\begin{equation} \label{alpha n1}
\zeta_A\left(q^n-1,(q-1)q^n\right)=-\left(\frac{\lambda_n\epsilon_1^{q^n}}{\lambda_{n+1}}+1\right)\zeta_A(q^{n+1}-1)
\end{equation}
and agrees with the above formula again in virtue of the fact that $\epsilon_1=-\lambda_1$.
\end{remark}

%%%%%%%%%%%%%

\section{Applications}

In this section, we consider Drinfeld modules on elliptic curves and ramifying hyperelliptic curves as studied in \cite{Chu20,GP18}.  We then derive a simple formula for the ratio $\alpha_{n,1}$. As a consequence, we prove a conjecture of Lara Rodr\'iguez and Thakur (see Theorem \ref{LRT conjecture}).

\subsection{Elliptic curves: the genus 1 case} \label{elliptic curves}

We work with Drinfeld modules on elliptic curves (i.e., $g=1$ and $d_\infty=1$) which were studied in detail in \cite{GP18}. We let $E$ be the elliptic curve defined over $\Fq$ given by
	\[ E: \eta^2+a_1 \theta \eta+a_3 \eta=\theta^3+a_2 \theta^2+a_4 \theta+a_6, \quad a_i \in \Fq. \]
We denote by $A=\Fq[\theta,\eta]$ the ring of functions on $E$ and $K$ its fraction field. We recall the $\Fq$-algebra homomorphism $\chi: A \to A \otimes \mathbb C_\infty$ given by $a \mapsto a \otimes 1$. We put $y=\chi(\eta)$ and $t=\chi(\theta)$.

The shtuka function attached to the standard sign-normalized Drinfeld module of rank one is given by
	\[ f=\frac{\nu}{\delta}=\frac{y-\eta-m(t-\theta)}{t-\alpha} \]
where $m=\frac{\eta-\beta^q}{\theta-\alpha^q}$ and $V=(\alpha,\beta) \in E(H)$ is the Drinfeld divisor. Recall that for all $i \geq 0$,
	\[ \epsilon_i=\frac{1}{f f^{(1)} \dots f^{(i-1)}} \Big|_{\Xi^{(i)}}, \quad \lambda_i=\frac{\delta^{(i+1)}}{\delta^{(1)} f^{(1)} \dots f^{(i)}} \Big|_{\Xi} \]
The first equality is proved by Thakur \cite[Proposition 0.3.6]{Tha93} as mentioned before. The second equality is shown by Anderson \cite[Proposition 0.3.8]{Tha93}, Green and Papanikolas \cite[Theorem 3.4 and Corollary 3.5]{GP18}.

By \eqref{alpha n1},
\begin{align*}
\alpha_{n,1} &=-\left(\frac{\lambda_n\epsilon_1^{q^n}}{\lambda_{n+1}}+1\right) \\
&= - \left(\frac{\delta^{(n+1)} f^{(n+1)}}{\delta^{(n+2)}} \Big|_\Xi \frac{1}{(f \big|_{\Xi^{(1)}})^{q^n}}+1 \right) \\
&= - \left(\frac{\nu^{(n+1)}}{\delta^{(n+2)}} \Big|_\Xi \frac{1}{(f \big|_{\Xi^{(1)}})^{q^n}}+1\right).
\end{align*}

\subsection{Ramifying hyperelliptic curves} \label{hyperelliptic curves}

In \cite{Chu20}, the first author extended the previous results to the case of ramifying hyperelliptic curves.  More precisely, letting $g \geq 1$ be an integer, we consider a ramifying hyperelliptic curve $X$ defined over $\Fq$ is given by 
	\[ X: \eta^2+F_2(\theta) \eta=F_1(\theta) \]
with $F_i \in \Fq[\theta]$, $F_1$ monic of degree $2g+1$ and $F_2$ of degree at most $g$. Thus the genus of $X$ equals $g$. We denote by $A=\Fq[\theta,\eta]$ the ring of functions on $X$ and $K$ its fraction field. We recall the $\Fq$-algebra homomorphism $\chi: A \to A \otimes \mathbb C_\infty$ given by $a \mapsto a \otimes 1$. We put $y=\chi(\eta)$ and $t=\chi(\theta)$.

By \cite[Proposition 3.1]{Chu20}, the shtuka function attached to the standard sign-normalized Drinfeld module of rank one can be expressed as
	\[ f=\frac{\nu}{\delta} \]
where $\nu=y-\eta-Q$ and $\delta, Q \in H[t]$ of degree $g$ in $t$. Further, we have $\widetilde{\sgn}(\nu)=\widetilde{\sgn}(\delta)=1$. Then by \cite[Proposition 0.3.6]{Tha93} and \cite[Proposition 3.3]{Chu20}, we still get
	\[ \epsilon_i=\frac{1}{f f^{(1)} \dots f^{(i-1)}} \Big|_{\Xi^{(i)}}, \quad \lambda_i=\frac{\delta^{(i+1)}}{\delta^{(1)} f^{(1)} \dots f^{(i)}} \Big|_{\Xi} \]
Thus we obtain a similar expression for $\alpha_{n,1}$: 
\begin{align*}
\alpha_{n,1} = - \left(\frac{\nu^{(n+1)}}{\delta^{(n+2)}} \Big|_\Xi \frac{1}{(f \big|_{\Xi^{(1)}})^{q^n}}+1\right).
\end{align*}

\subsection{Class number one base rings} 

Following Thakur \cite{Tha93} (see also \cite{LRT20}), another important class of base rings consists of class number one rings, i.e., we require that $A$ is a principal ideal domain. There are exactly four of them listed as Examples (i)-(iv) in \cite{LRT20}:
\begin{enumerate}
\item[(i)] $A = \bb F_2[\theta,\eta] / (\eta^2 + \eta + \theta^3 + \theta + 1)$;

\item[(ii)] $A = \bb F_3[\theta,\eta] / (\eta^2  - (\theta^3 - \theta - 1))$;

\item[(iii)] $A = \bb F_4[\theta,\eta] / (\eta^2 + \eta + \theta^3 + w)$, where $w^2+w+1=0$, i.e., $w$ is a generator of $\bb F_4^\times$;

\item[(iv)] $A = \bb F_2[\theta,\eta] / (\eta^2 + \eta + \theta^5 + \theta^3 + 1)$.

\end{enumerate}

They considered the following ratios (in \cite{LRT20}, these ratios are denoted by $R_n$)
\[
\alpha_{n,1} := \frac{\zeta_A(q^n-1,q^{n+1}-q^n)}{\zeta_A(q^{n+1}-1)},
\]
and conjectured that they are in $K$ in the above four cases. In each of the four cases, we will set $X$ and $Y$ to be some $q^n$-th power of $\theta$ and $\eta$ respectively.

\begin{conjecture}[\cite{LRT20}, Conjecture 3.3] \label{conjLRT} 
The following formulas hold in Examples (i)-(iv) above:
\begin{enumerate}
\item[(i)] Set $X := \theta^{2^n}$ and $Y := \eta^{2^n}$. Then
\begin{equation} \label{eq:RTEg1}
\begin{split}
\alpha_{n,1} &= \frac{\theta^{2^{n+1}}+\theta^2}{\eta^{2^{n+1}}+\eta+\theta^{2^{n+1}+1}+\theta} \\
&= \frac{X^2+\theta^2}{Y^2+\eta+\theta X^2+\theta}.
\end{split}
\end{equation}

\item[(ii)] Set $X := \theta^{3^n}$ and $Y := \eta^{3^n}$. Then
\begin{equation} \label{eq:RTEg2}
\begin{split}
\alpha_{n,1} &= \frac{(\theta^{3^n}-\theta) (\eta^{3^n}-\eta)^2 + (\theta^{3^n}-\theta) (-\theta^{3^n}-\theta^3-\theta+1) }{\theta^{2 \cdot 3^{n+1}} + \theta^{3^{n+1}+1} + \theta^{3^{n+1}} + \eta^{3 \cdot 3^n + 1} + \theta^2-\theta+1} \\
&= \frac{(X-\theta) (Y-\eta)^2 - X^2 + (-\theta^3 + 1) X + \theta^2 - \theta }{\eta Y^3 + X^6 + (\theta+1) X^3 + \theta^2-\theta+1}.
\end{split}
\end{equation}

\item[(iii)] Set $X := \theta^{4^n}$ and $Y := \eta^{4^n}$. Then 
\begin{equation} \label{eq:RTEg3}
\begin{split}
\alpha_{n,1} &= \frac{(\theta^{4^n}+\theta)(\eta^{4^{n+1}}+\eta^4) + (\theta^{4^{n+1}} + \theta^4) (\theta^{4^n+2} + \theta^3+1) + (\theta^{4^n}+\theta)}{\theta^{2 \cdot 4^{n+1}+2} + (\theta^{4^{n+1}} + \theta) (\eta^{4^{n+1}} + \eta) + \theta^{4^{n+1}}} \\
&= \frac{(X+\theta)(Y^4+\eta^4) + (X^4 + \theta^4) (\theta^{2} X + \theta^3+1) + X+\theta}{\theta^2 X^8 + (X^4 + \theta)(Y^4 + \eta) + X^4} \\
&= \frac{(X+\theta)(Y^4+\eta^4) + \theta^2 X^5 + (\theta^3 + 1) X^4 + (\theta^6 + 1) X + \theta^7 + \theta^4 + \theta}{(X^4 + \theta)(Y^4 + \eta) + \theta^2 X^8 + X^4}. 
\end{split}
\end{equation}

\item[(iv)] Set $X := \theta^{2^{n-1}}$ and $Y := \eta^{2^{n-1}}$. Then
\begin{equation} \label{eq:RTEg4}
\begin{split}
\alpha_{n,1} &= \frac{1}{X^{24} + \theta X^{16} + (\theta+1) X^8 + \theta^2 + \theta} \bigg( X^{22} + (1 + \theta) \\
&\phantom{==} \cdot (X^{20} + X^{18} + X^{16}) + (\theta^2+\theta+1)(X^{12} + X^{10}) + X^9 + \theta X^8 \\
&\phantom{==} + X^5 + (Y + \eta)(X^4 + X^2) + \theta^2 + \theta \bigg).
\end{split}
\end{equation}
\end{enumerate}
\end{conjecture}

\begin{theorem} \label{LRT conjecture}
Conjecture \ref{conjLRT} holds.
\end{theorem}

\begin{proof}
In the case of elliptic curve and ramifying hyperelliptic curves, we have proved an expression in terms of the shtuka function $f$ for the standard sign-normalized rank one Drinfeld module $\phi$:
\[
\alpha_{n,1} = -\left( \left. \frac{\nu^{(n+1)}}{\delta^{(n+2)}} \right|_{\Xi} \frac{1}{(f|_{\Xi^{(1)}})^{q^n}} +1 \right).
\]
Here $\nu, \delta$ are respectively the numerator and denominator of the shtuka function $f$, as given in \S \ref{elliptic curves} and \S \ref{hyperelliptic curves}. Upon evaluation and reusing the notations from Conjecture \ref{conjLRT}, we have
\begin{enumerate}
\item[(i)] 
\[
f = \frac{y - \eta - \theta(t - \theta)}{t - \theta - 1}, \qquad f|_{\Xi^{(1)}} = 1,
\]
\begin{equation} \label{eq:CNPEg1}
\begin{split}
\alpha_{n,1} &= \frac{\eta + \eta^{2^{n+1}} + \theta^{2^{n+1}+1} + \theta + 1}{\theta^{2^{n+2}} +\theta+1} \\
&= \frac{Y^2 + \eta + \theta X^2 + \theta + 1}{X^4 +\theta+1}.
\end{split}
\end{equation}

\item[(ii)] 
\[
f = \frac{y - \eta - \eta(t - \theta)}{t - \theta - 1}, \qquad f|_{\Xi^{(1)}} = \frac{1}{\eta},
\]
\begin{equation} \label{eq:CNPEg2}
\begin{split}
\alpha_{n,1} &= \frac{(\eta - \eta^{3^{n+1}} - \eta^{3^{n+1}} (\theta - \theta^{3^{n+1}}) ) \eta^{3^n} - \theta^{3^{n+2}} + \theta - 1 }{\theta^{3^{n+2}} - \theta + 1} \\
&= \frac{Y^4 (X^3 - \theta - 1) + \eta Y - X^9 + \theta - 1}{X^9 - \theta + 1}.
\end{split}
\end{equation}

\item[(iii)] 
\[
f = \frac{y - \eta - \theta^2(t - \theta)}{t - \theta}, \qquad f|_{\Xi^{(1)}} = \frac{1}{\theta^4+\theta},
\]
\begin{equation} \label{eq:CNPEg3}
\begin{split}
\alpha_{n,1} &= \frac{ (\eta^{4^{n+1}} + \eta + \theta^{2 \cdot 4^{n+1}} (\theta + \theta^{4^{n+1}})) (\theta^{4^{n+1}} + \theta^{4^n}) + \theta^{4^{n+2}} + \theta}{\theta^{4^{n+2}} + \theta} \\
% &= \frac{ (\eta^{4^{n+1}} + \eta)(\theta^{4^{n+1}} + \theta^{4^n}) + \theta^{13 \cdot 4^{n}} + \theta^{12 \cdot 4^{n} + 1} + \theta^{9 \cdot 4^n + 1} + \theta }{\theta^{4^{n+2}} + \theta}.
&= \frac{ (Y^4 + \eta)(X^4 + X) + X^{13} + \theta X^{12} + \theta X^{9} + \theta }{X^{16} + \theta}.
\end{split}
\end{equation}

\item[(iv)] 
\[
f = \frac{y + \eta - (t+\theta)(\theta^4 + \theta^3 + \theta^2(t+1))}{\theta^3 + t\theta^2 + (1+t) \theta + t^2+t}, \qquad f|_{\Xi^{(1)}} = \frac{1}{\theta^2+\theta},
\]
\begin{equation} \label{eq:CNPEg4}
\begin{split}
\alpha_{n,1} &= \frac{1}{\theta^{3 \cdot 2^{n+2}} + \theta^{2 \cdot 2^{n+2}+1} + (1 + \theta) \theta^{2^{n+2}} + \theta^2 + \theta} \bigg( \big( \eta + \eta^{2^{n+1}} \\ 
&\phantom{==} + (\theta + \theta^{2^{n+1}}) (\theta^{4 \cdot 2^{n+1}} + \theta^{3 \cdot 2^{n+1}} + \theta^{2 \cdot 2^{n+1}} (\theta + 1)) \big) \cdot (\theta^{2^{n+1}} + \theta^{2^n}) \\
&\phantom{==} + (\theta^{3 \cdot 2^{n+2}} + \theta^{2 \cdot 2^{n+2}+1} + (1 + \theta) \theta^{2^{n+2}} + \theta^2 + \theta) \bigg) \\
&= \frac{1}{X^{24} + \theta X^{16} + (\theta+1) X^8 + \theta^2 + \theta} \bigg( \big( \eta + Y^4 \\
&\phantom{==} + (\theta + X^4) (X^{16} + X^{12} + (\theta + 1) X^8) \big) \cdot (X^4 + X^2) \\ 
&\phantom{==} + (X^{24} + \theta X^{16} + (\theta+1) X^8 + \theta^2 + \theta) \bigg) \\
&= \frac{1}{X^{24} + \theta X^{16} + (\theta+1) X^8 + \theta^2 + \theta} \bigg( (\eta + Y^4)(X^4+X^2) \\
&\phantom{==} + X^{22} + (\theta+1) (X^{20} + X^{18}) + (\theta+1) X^{16} + X^{14} \\
&\phantom{==} + (\theta^2+\theta) (X^{12} + X^{10}) + (\theta+1) X^8 + \theta^2 + \theta \bigg).
\end{split}
\end{equation}

\end{enumerate}

We now prove that our expressions match those in Conjecture \ref{conjLRT}. 

\medskip
\noindent {\bf Example (i) \texorpdfstring{$A=\bb F_2[\theta,\eta] / (\eta^2 + \eta + \theta^3 + \theta + 1)$.}{g = 1, h = 1, F2}}

Set (\ref{eq:RTEg1}) and (\ref{eq:CNPEg1}) to be equal and cross multiply. We want to show that the following equality holds:
\begin{equation*}
Y^4 + \eta^2 + \theta^2 X^4 + \theta^2 + Y^2 + \eta + \theta X^2 + \theta = X^6 + \theta^2 X^4 + (\theta+1) X^2 + \theta^3 + \theta^2.  
\end{equation*}
Using $\eta^2 + \eta = \theta^3 + \theta + 1$, we can simplify the above equation as
\begin{equation*}
Y^4 + Y^2 = X^6 + X^2 + 1.
\end{equation*}
This is true since 
\[
Y^2 + Y = \eta^{2 \cdot 2^n} + \eta^{2^n} = (\eta^2 + \eta)^{2^n} = (\theta^3 + \theta + 1)^{2^n} = X^3 + X + 1,
\]
and we obtain the desired equality by squaring this equation. This proves that the expressions (\ref{eq:CNPEg1}) and (\ref{eq:RTEg1}) are equal.

% For all $n \geq 1$, we set $z_n:=\eta+\eta^{2^n}$. We claim that
% 	\[ z_n^2+z_n=\theta^{2^{n+1}+2^n}+\theta^{2^n}+\theta^3+\theta.\]
% In fact, 
% \begin{align*}
% z_n^2+z_n &=\eta^2+\eta^{2^{n+1}}+\eta+\eta^{2^n} \\
% &=(\eta^2+\eta)+(\eta^2+\eta)^{2^n} \\
% &=(\theta^3 + \theta + 1)+(\theta^3 + \theta + 1)^{2^n} \\
% &=\theta^{2^{n+1}+2^n}+\theta^{2^n}+\theta^3+\theta.
% \end{align*}

%%%%%%%%%%

\medskip
\noindent {\bf Example (ii) \texorpdfstring{$A=\bb F_3[\theta,\eta] / (\eta^2  - (\theta^3 - \theta - 1))$.}{g = 1, h = 1, F3}}

By squaring $\eta^2 = \theta^3 - \theta - 1$, we have $\eta^4 = \theta^6 + \theta^4 + \theta^3 + \theta^2 - \theta + 1$. Hence 
$$
Y^4 = (\eta^4)^{3^n} = X^6 + X^4 + X^3 + X^2 - X + 1.
$$
Thus we can rewrite the expression in (\ref{eq:CNPEg2}) as 
\begin{multline} \label{eq:CNPEg2Rewrite}
\alpha_{n,1} = \frac{1}{X^9 - \theta + 1} \big(\eta Y + X^7 - \theta X^6 + X^5 \\
+ (-\theta+1) X^4 - \theta X^3 - (\theta+1) X^2 + (\theta + 1) X + 1 \big).
\end{multline}

We also expand the terms involving $Y$ in \eqref{eq:RTEg2}. First, we have that 
\[ Y^2 = (\eta^2)^{3^n} = (\theta^3 - \theta - 1)^{3^n} = X^3 - X - 1. \]
Thus
\begin{equation*}
\begin{split}
(Y-\eta)^2 &= Y^2 + \eta Y + \eta^2 \\
&= (X^3 - X - 1) + \eta Y + (\theta^3 - \theta - 1) \\
&= \eta Y + X^3 - X + \theta^3 - \theta + 1.
\end{split}
\end{equation*}
At the same time, 
$$
Y^3 = Y^2 \cdot Y = (X^3 - X - 1)Y.
$$
Putting these into \eqref{eq:RTEg2}, we have that the expression equals 
\begin{equation*}
\begin{split}
&\phantom{=1} \frac{(X-\theta)(\eta Y + X^3 - X + \theta^3 - \theta + 1) - X^2 + (-\theta^3 + 1) X + \theta^2 - \theta }{(X^3-X-1) \eta Y + X^6 + (\theta+1) X^3 + \theta^2-\theta+1} \\
&= \frac{(X - \theta) \eta Y + X^4 - \theta X^3 + X^2 - X - \theta^2 + \theta}{(X^3-X-1) \eta Y + X^6 + (\theta+1) X^3 + \theta^2-\theta+1} .
\end{split}
\end{equation*}

Set this and \eqref{eq:CNPEg2Rewrite} to be equal and cross multiply, we want to show that the following equality holds:
\begin{equation*}
\begin{split}
&\phantom{=1} \big( X^{10} - \theta X^9 + (-\theta + 1) X + \theta^2 - \theta \big) \eta Y + X^{13} - \theta X^{12} + X^{11} \\
&\phantom{=1} - X^{10} + (-\theta^2+\theta) X^9 + (-\theta+1) X^4 + (\theta^2 - \theta) X^3 \\
&\phantom{=1} + (-\theta+1) X^2 + (\theta-1) X + \theta^3 + \theta^2 + \theta \\
&= \eta^2 (X^3-X-1) Y^2 + \big( X^{10} - \theta X^9 + (-\theta+1) X + \theta^2 - \theta \big) \eta Y \\
&\phantom{=1} + X^{13} - \theta X^{12} + X^{11} - X^{10} + (-\theta^2+\theta) X^9 + (-\theta^3 + \theta + 1) X^6 \\
&\phantom{=1} - (\theta^3 + 1) X^4 + (-\theta^3 + \theta^2 + 1) X^3 + (-\theta^3 - 1) X^2 + (\theta^3 + 1) X \\
&\phantom{=1} + \theta^2 - \theta + 1.
\end{split}
\end{equation*}
% \begin{equation*}
% \begin{split}
% % &\phantom{==1} \\
% \bigg( X^{10} - \theta X^9 + (-\theta + 1) X + \theta^2 &\phantom{==1} \eta^2 (X^3-X-1) Y^2 + \bigg( X^{10} - \theta X^9 \\
% - \theta \bigg) \eta Y + X^{13} - \theta X^{12} + X^{11} &\phantom{==1} + (-\theta+1) X + \theta^2 - \theta \bigg) \eta Y \\
% - X^{10} + (-\theta^2+\theta) X^9 &\phantom{1}=\phantom{1} + X^{13} - \theta X^{12} + X^{11} - X^{10} \\
% + (-\theta+1) X^4 + (\theta^2 - \theta) X^3 &\phantom{==1} + (-\theta^2+\theta) X^9 + (-\theta^3 + \theta + 1) X^6 \\
% + (-\theta+1) X^2 + (\theta-1) X &\phantom{==1} - (\theta^3 + 1) X^4 + (-\theta^3 + \theta^2 + 1) X^3 \\
% + \theta^3 + \theta^2 + \theta &\phantom{==1} + (-\theta^3 - 1) X^2 + (\theta^3 + 1) X \\
% &\phantom{==1} + \theta^2 - \theta + 1
% % &\phantom{==1}  \\
% \end{split}.
% \end{equation*}

Cancelling common terms and rearranging, we want the following equality to hold: 
\begin{equation*}
\begin{split}
\phantom{=1} \eta^2 (X^3 - X - 1) Y^2 
&= (\theta^3 - \theta - 1) X^6 + (\theta^3 - \theta - 1) X^4 + (\theta^3 - \theta - 1) X^3 \\
&\phantom{=1} + (\theta^3 - \theta - 1) X^2 - (\theta^3 - \theta - 1) X + (\theta^3 - \theta - 1).
\end{split}
\end{equation*}
This is indeed true since $\eta^2 = \theta^3 - \theta - 1$, and 
$$
(X^3-X-1) Y^2 = (X^3-X-1)^2 = X^6 + X^4 + X^3 + X^2 - X + 1.
$$
This proves that the expressions (\ref{eq:CNPEg2}) and (\ref{eq:RTEg2}) are equal.

%%%%%%%%%%

\medskip
\noindent {\bf Example (iii) \texorpdfstring{$A=\bb F_4[\theta,\eta] / (\eta^2 + \eta + \theta^3 + w)$ where $w^2+w+1=0$.}{g = 1, h = 1, F4}}

% The proof is similar to Example (i). To ease notation, we rewrite (\ref{eq:CNPEg3}) in terms of $X := \theta^{4^n}$ and $Y := \eta^{4^n}$. 
% \begin{equation*}
% % \begin{split}
% \alpha_{n,1} = \frac{(Y^4 + \eta)(X^4 + X) + X^{13} + \theta X^{12} + \theta X^9 + \theta}{X^{16} + \theta}.
% % \end{split}
% \end{equation*}

Set the expressions (\ref{eq:RTEg3}) and (\ref{eq:CNPEg3}) to be equal and cross multiply. We want to show that the following equality holds:
\begin{equation*}
\begin{split}
&\phantom{=1} (X^{17} + \theta X^{16} + \theta X + \theta^2) (Y^4 + \eta^4) + \theta^2 X^{21} + (\theta^3 + 1) X^{20} + (\theta^6 + 1) X^{17} \\
&\phantom{=1} + (\theta^7 + \theta^4 + \theta) X^{16} + \theta^3 X^5 + (\theta^4 + \theta) X^4 + (\theta^7 + \theta) X + \theta^8 + \theta^5 + \theta^2 \\
&= (X^8 + X^5 + \theta X^4 + \theta X) (Y^4 + \eta)^2 + \big(X^{17} + \theta X^{16} + X^8 + X^5 \\
&\phantom{=1} + \theta X^4 + \theta^2 \big) (Y^4 + \eta) + \theta^2 X^{21} + \theta^3 X^{20} + (\theta^3+1) X^{17} + \theta X^{16} \\
&\phantom{=1} + \theta X^{13} + \theta^3 X^8 + \theta X^4.
\end{split}
\end{equation*}
% \begin{equation*}
% \begin{split}
% (X^{17} + \theta X^{16} + \theta X + \theta^2) (Y^4 + \eta^4) &\phantom{==1} (X^8 + X^5 + \theta X^4 + \theta X) (Y^4 + \eta)^2 \\
% + \theta^2 X^{21} + (\theta^3 + 1) X^{20} + (\theta^6 + 1) X^{17} &\phantom{==1} + \bigg(X^{17} + \theta X^{16} + X^8 + X^5  \\
% + (\theta^7 + \theta^4 + \theta) X^{16} + \theta^3 X^5 + (\theta^4 + \theta) X^4 &\phantom{1}=\phantom{1} + \theta X^4 + \theta^2 \bigg) (Y^4 + \eta) + \theta^2 X^{21} \\
% + (\theta^7 + \theta) X + \theta^8 + \theta^5 + \theta^2 &\phantom{==1} + \theta^3 X^{20} + (\theta^3+1) X^{17} + \theta X^{16} \\
% &\phantom{==1} + \theta X^{13} + \theta^3 X^8 + \theta X^4 \\
% % &\phantom{==1} 
% \end{split}.
% \end{equation*}

Recall that $\eta^2 = \eta + \theta^3 + w$. Squaring both sides, we have that 
\[
\eta^4 = \eta^2 + \theta^6 + w^2 = \eta + \theta^6 + \theta^3 + w^2 + w
= \eta + \theta^6 + \theta^3 + 1.
\]
Using this and rearranging the long equation, we want to show that the following equality holds:
\begin{equation*}
\begin{split}
&\phantom{=1} (X^8 + X^5 + \theta X^4 + \theta X) (Y^4 + \eta)^2 + (X^8 + X^5 + \theta X^4 + \theta X) (Y^4 + \eta) \\
&= (X^{17} + \theta X^{16} + \theta X + \theta^2) (\theta^6 + \theta^3 + 1) + X^{20} + (\theta^6 + \theta^3) X^{17} + (\theta^7 + \theta^4) X^{16} \\
&\phantom{=1} + \theta X^{13} + \theta^3 X^8 + \theta^3 X^5 + \theta^4 X^4 + (\theta^7 + \theta) X + \theta^8 + \theta^5 + \theta^2.
\end{split}
\end{equation*}
% \begin{equation*}
% \begin{split}
% (X^8 + X^5 + \theta X^4 + \theta X) (Y^4 + \eta)^2 &\phantom{==1} (X^{17} + \theta X^{16} + \theta X + \theta^2) (\theta^6 + \theta^3 + 1) \\
% + ( X^8 + X^5 + \theta X^4 + \theta X ) (Y^4 + \eta) &\phantom{1}=\phantom{1} + X^{20} + (\theta^6 + \theta^3) X^{17} + (\theta^7 + \theta^4) X^{16} \\
% &\phantom{==1} + \theta X^{13} + \theta^3 X^8 + \theta^3 X^5 + \theta^4 X^4 \\
% &\phantom{==1} + (\theta^7 + \theta) X + \theta^8 + \theta^5 + \theta^2 \\
% % &\phantom{1}=\phantom{1} \\
% % &\phantom{==1} \\
% % &\phantom{==1} 
% \end{split}.
% \end{equation*}

The right hand side can be simplified to 
\begin{equation*}
\begin{split}
&\phantom{=1} X^{20} + X^{17} + \theta X^{16} + \theta X^{13} + \theta^3 X^8 + \theta^3 X^5 + \theta^4 X^4 + \theta^4 X \\
&= (X^8 + X^5 + \theta X^4 + \theta X)(X^{12} + \theta^3)
\end{split}
\end{equation*}
Hence, it suffices to show that 
\[
(Y^4 + \eta)^2 + (Y^4 + \eta) = X^{12} + \theta^3.
\]
Raising $\eta^2 + \eta = \theta^3 + w$ by $4 \cdot 4^n$-th power, we have that 
\[
Y^8 + Y^4 = X^{12} + w.
\]
In particular, $w^4 = w$ as $w \in \bb F_4$. Thus, 
\[
(Y^4 + \eta)^2 + (Y^4 + \eta) 
= Y^8 + Y^2 + \eta^2 + \eta 
= X^{12} + w + \theta^3 + w = X^{12} + \theta^3,
\]
completing the proof. 

% Set $z_n := \eta^{4^n} + \eta$. Equivalently, we want to show that for $n \geq 0$,
% \[
% z_n^2 + z_n = \theta^{3 \cdot 4^n} + \theta^3.
% \]
% We proceed by induction. When $n=0$ both sides are 0. For the inductive step, recall from above that $\eta^4 = \eta^2 + \theta^6 + w^2 = \eta + \theta^6 + \theta^3 + 1$. Hence 
% \[
% z_{n+1} = z_n^4 + \eta^4 + \eta = z_n + \theta^{6 \cdot 4^n} + \theta^{3 \cdot 4^n} + 1.
% \]
% Squaring both sides and adding the same equation, we have 
% \[
% z_{n+1}^2 + z_{n+1} = z_n^2 + z_n + \theta^{12 \cdot 4^n} + \theta^{3 \cdot 4^n}
% = \theta^{3 \cdot 4^{n+1}} + \theta^3,
% \]
% with the last equality being true by the inductive hypothesis. This completes the induction, as well as the showing that expressions (\ref{eq:CNPEg3}) and (\ref{eq:RTEg3}) are equal. 

%%%%%%%%%%

\medskip
\noindent {\bf Example (iv) \texorpdfstring{$A=\bb F_2[\theta,\eta] / (\eta^2 + \eta + \theta^5 + \theta^3 + 1)$.}{g = 2, h = 1, F2}}

% This one is the easiest among the four examples. Using the notation $X = \theta^{2^{n-1}}$ and $Y^{2^{n-1}}$, we can rewrite the expression in (\ref{eq:CNPEg4}) as
% \begin{equation*}
% \begin{split}
% \alpha_{n,1} &= \frac{1}{X^{24} + \theta X^{16} + (\theta+1) X^8 + \theta^2 + \theta} \bigg( \big( \eta + Y^4 \\
% &\phantom{==} + (\theta + X^4) (X^{16} + X^{12} + (\theta + 1) X^8) \big) \cdot (X^4 + X^2) \\ 
% &\phantom{==} + (X^{24} + \theta X^{16} + (\theta+1) X^8 + \theta^2 + \theta) \bigg) \\
% &= \frac{1}{X^{24} + \theta X^{16} + (\theta+1) X^8 + \theta^2 + \theta} \big( (\eta + Y^4)(X^4+X^2) \\
% &\phantom{==} + X^{22} + (\theta+1) (X^{20} + X^{18}) + (\theta+1) X^{16} + X^{14} \\
% &\phantom{==} + (\theta^2+\theta) (X^{12} + X^{10}) + (\theta+1) X^8 + \theta^2 + \theta \big) 
% \end{split}
% \end{equation*}

This one is the easiest among the four examples. The expression (\ref{eq:CNPEg4}) has the same denominator as Lara Rodr\'iguez-Thakur's expression (\ref{eq:RTEg4}). Hence it suffices to show that their numerators are the same. By setting the numerators to be equal, it boils down to showing that 
\begin{equation*}
Y^4 + Y = X^{10} + X^6 + X^5 + X^3,
\end{equation*}
Or equivalently,
\begin{equation*}
\eta^{2^{n+1}} + \eta^{2^{n-1}} = \theta^{10 \cdot 2^{n-1}} + \theta^{6 \cdot 2^{n-1}} + \theta^{5 \cdot 2^{n-1}} + \theta^{3 \cdot 2^{n-1}}
\end{equation*}
for $n \geq 1$. This follows from the equality
\begin{equation*}
\eta^4 + \eta = \theta^{10} + \theta^6 + \theta^5 + \theta^3
\end{equation*}
by squaring $\eta^2 = \eta + \theta^5 + \theta^3 + 1$. The proof is complete. 
\end{proof}

%%%%%%%%%%%%%%%%%%%%

\end{document}